\documentclass[11pt,reqno]{amsart}
\usepackage{xifthen,setspace}
\usepackage{bbm}
\usepackage{esvect}
\usepackage{amsmath} 
\usepackage{style}

\newcommand{\hf}{{H_{\beta,h}}}
\renewcommand{\gf}{{G_{\beta,h}}}
\newcommand{\f}{f_{\beta,h}}

\newcommand{\eps}{{\epsilon}}

\newcommand{\p}{{\P}}
\newcommand{\e}{{\E}}

\newcommand{\cs}{{\bar{\bm X}_{N}}}

\newcommand{\Xrb}{{\bar{X}_{\cdot r}}}
\newcommand{\Xsb}{{\bar{X}_{\cdot s}}}

\allowdisplaybreaks

\title[Berry-Esseen Bounds in the Tensor Curie-Weiss Potts Model]{Rates of Convergence of the Magnetization in the Tensor Curie-Weiss Potts Model}

\author[Bhowal]{Sanchayan Bhowal}
\address{Statistics and Mathematics Unit, Indian Statistical Institute, Bangalore, India, {\tt sanchayan.bhowal2509@gmail.com}}

\author[Mukherjee]{Somabha Mukherjee} 
\address{Department of Statistics and Data Science, National University of Singapore, Singapore. {\tt somabha@nus.edu.sg}}

\begin{document}

\begin{abstract}
    In this paper, we derive distributional convergence rates for the magnetization vector and the maximum pseudolikelihood estimator of the inverse temperature parameter in the tensor Curie-Weiss Potts model. Limit theorems for the magnetization vector have been derived recently in \cite{bhowal_mukh}, where several phase transition phenomena in terms of the scaling of the (centered) magnetization and its asymptotic distribution were established, depending upon the position of the true parameters in the parameter space. In the current work, we establish Berry-Esseen type results for the magnetization vector, specifying its rate of convergence at these different phases. At ``most" points in the parameter space, this rate is $N^{-1/2}$ ($N$ being the size of the Curie-Weiss network), while at some \textit{special} points, the rate is either $N^{-1/4}$ or $N^{-1/6}$, depending upon the behavior of the fourth derivative of a certain \textit{negative free energy function} at these special points. These results are then used to derive Berry-Esseen type bounds for the maximum pseudolikelihood estimator of the inverse temperature parameter whenever it lies above a certain criticality threshold.
\end{abstract}



\maketitle
\section{Introduction}

The Potts model \cite{fywu}, originally named after Renfrey Potts \cite{pottsorig}, is a generalization of the Ising model \cite{isingorig}, where the \textit{spin} of any particular site can have more than two states, each such state being referred to as a color. This model is immensely useful in explaining diverse physical phenomena such as magnetism, phase transitions, and social behavior, and has found widespread applications in a number of different fields such as biomedical problems \cite{cellular, gene77}, image processing and computer vision \cite{impro, impro2}, spatial statistics \cite{spatstat}, social sciences \cite{socialsci} and finance \cite{fina66,bornholdt}. Although the classical Potts model captures only pairwise interactions between the sites of a network, in many scientific and real life contexts, such as the atomic interactions on a crystal surface or the peer group effects in a social circle, multibody interactions are more common. A natural extension of the classical Potts model which can also capture higher order interactions, is the tensor Potts model \cite{bhowal_mukh}, where the sufficient statistic is a multi-linear form of the indicators of monochromatic site tuples. Tensor versions of the closely related Ising model have also emerged in a number of recent works such as \cite{eichcubic, contucci, Godwin, smfl, smmpl}.

One can think of tensor Potts models as Potts models on hypergraphs. However, studying asymptotics of the sufficient statistics for Potts models on arbitrary hypergraphs is hopelessly challenging, unless one works with simpler interaction structures, such as assuming that all the tuples of nodes of a particular order interact with the same strength. The underlying hypergraph in this case is complete, and the resulting model is known as the (tensor) Curie-Weiss Potts model. Asymptotics of the sufficient statistics (magnetization vector) in the tensor Curie-Weiss Potts model were established in \cite{bhowal_mukh}, where the authors established several interesting phase transition phenomena. In particular, three different rates of convergence ($N^{-1/2}, N^{-1/4}$ and $N^{-1/6}$) of the magnetization vector may arise depending on the positions of the true parameters in the parameter space. Further, the nature of the limiting distribution is also different on a critical curve and its boundary point lying in the interior of the parameter space, from the rest of the space. In the present work, we aim to establish speeds of convergence of the distributions of the magnetization vector to their corresponding asymptotic distribution, in the form of Berry-Esseen type bounds. This problem was solved for the classical two-spin Curie-Weiss Potts model in \cite{eichelsbacher2015rates,bmart}, for the closely related two-spin Curie-Weiss Ising model in \cite{chenshao, moderate_deviation_Can, be_eich}, and more recently, for the tensor Curie-Weiss Ising model in \cite{smstein}. Following the technique in \cite{eichelsbacher2015rates}, we will use Stein's method of exhangeable pairs \cite{reinert2009multivariate, stein2004use} to derive our bounds. In summary, we show that the Berry-Esseen bound for distributional convergence of the magnetization vector to its corresponding limiting distribution is of the order $\log N/\sqrt{N}$ at \textit{most} parameter points, $N^{-1/4}$ at some special points of one type, and $N^{-1/6}$ at some other special points of a different type. We also use these results to derive Berry-Esseen bounds of order $\log N /\sqrt{N}$ for the maximum pseudolikelihood estimator of the inverse temperature parameter in the tensor Curie-Weiss Potts model, whenever this parameter lies above a certain criticality threshold.

The rest of the paper is organized as follows. In Section \ref{sec:prelim}, we outline some preliminary definitions and techniques required for our analysis. This is followed by Section \ref{sec:main_res}, where we give the main results of this paper. The proofs of the main results are given in Section \ref{sec:proof}. Finally, proofs of some technical lemmas necessary for proving the main results are given in the appendix.

\section{Preliminaries}\label{sec:prelim}
\subsection{Model Description} For integers $p\ge 2$ and $q\ge 2$, the $p$-tensor Potts model is a discrete probability distribution on the set $[q]^N$ (here and afterwards, for a positive integer $m$, we will use $[m]$ to denote the set $\{1,2,\ldots,m\}$)  for some positive integers $q$ and $N$, given by:
\begin{equation}\label{gp}
    \p_{\beta,h,N}(\bm X) := \frac{1}{q^N Z_N(\beta,h)} \exp \left(\beta \sum_{1\le i_1,\ldots,i_p\le N} J_{i_1,\ldots,i_p}\mathbbm{1}_{X_{i_1}=\cdots=X_{i_p}} + h\sum_{i=1}^N\mathbbm{1}_{X_i=1}\right) \quad(\bm X \in [q]^N)~,
\end{equation}
where $\beta>0$, $h \geq 0$ and $\bm J := ((J_{i_1, \ldots,i_p}))_{i_1,\ldots,i_p\in [N]}$ is a symmetric tensor. The $p$-tensor Curie-Weiss Potts model is obtained by taking $J_{i_1,\ldots,i_p} := N^{1-p}$ for all $(i_1,\ldots,i_p) \in [N]^p$, whence model \eqref{gp} takes the form:
\begin{equation}\label{eq:cp}
    \p_{\beta,h,N}(\bm X) := \frac{1}{q^N Z_N(\beta,h)} \exp \left(\beta N\sum_{r=1}^q \Xrb^p + Nh\bar{X}_{\cdot 1}\right) \quad(\bm X \in [q]^N),
\end{equation}
where $\Xrb := N^{-1} \sum_{i=1}^N X_{i,r}$ with $X_{i,r}:= \mathbbm{1}_{X_i=r}$. The variables $p$ and $q$ are called the \textit{interaction order} and the \textit{number of states/colors} of the Potts model, respectively. A sufficient statistic for the exponential family \eqref{eq:cp} is the empirical magnetization vector:
$$\cs := \left(\bar{X}_{\cdot 1},\ldots,\bar{X}_{\cdot q}\right)^\top~. $$
Note that $\cs$ is a probability vector, i.e. has non-negative entries adding to $1$. A complete description of the asymptotics of $\cs$ on the entire parameter space:
$$\Theta := \{(\beta,h): \beta>0, h \geq 0\} = (0,\infty)\times [0,\infty)$$ was given in \cite{bhowal_mukh}, where it was shown that the nature of the asymptotics depends upon the maximizer(s) of a certain \textit{negative free energy} function, and the behavior of this function at the maximizer(s). We summarize these concepts in the next section.

\subsection{Partitioning the Parameter Space}
We now recall some preliminaries introduced in \cite{bhowal_mukh}. For $p,q\ge 2$ and $(\beta,h) \in \Theta$, the \textit{negative free energy} function $H_{\beta,h}: \cP_q \to \R$ is defined as:
\begin{equation*}
    H_{\beta,h}(\bm t) := \beta \sum_{r=1}^q t_r^p + ht_1 - \sum_{r=1}^q t_r \log t_r,
\end{equation*}
where $\cP_q$ denotes the set of all $q$-dimensional probability vectors.

Let us define another function $\gf: \cP_q \to \R$,
\begin{equation*}
    \begin{aligned}
        G_{\beta, h}(\bm x)=\beta(p-1) \sum_{r=1}^{q} x_{r}^{p}-\log \left(\sum_{r=1}^{q} \exp \left(p \beta x_{r}^{p-1}+h \delta_{r, 1}\right)\right)
    \end{aligned}
\end{equation*}
where $\delta_{i,j} := \mathbbm{1}_{i=j}$. We will see later that the maximizer(s) of $H_{\beta,h}$ are minimizers of $G_{\beta,h}$ and vice versa.
It follows from Proposition F.1 in \cite{bhowal_mukh} that the global maximizers of $H_{\beta,h}$ can be parametrized as permutations of the vector
\begin{equation}
    \label{xsdefine}
    \bm x_s=\left(\frac{1+(q-1)s}{q},\frac{1-s}{q},\ldots,\frac{1-s}{q}\right),
\end{equation}
for some $s \in [0,1)$, and hence, the problem of maximizing $H_{\beta,h}$ can be reduced to a one-dimensional optimization of the function $f_{\beta,h}(s) : =H_{\beta,h}(\bm x_s)$. Observe that the map $s\mapsto \bm x_s$ is one-one, since $s = 1-q x_{s,2}$.
We see that:
$$f_{\beta,h}(s)=(q-1)k\left(\frac{1-s}{q}\right)+k\left(\frac{1+(q-1)s}{q}\right) +\left(\frac{1+(q-1)s}{q}\right) h,$$
where $k(x)= k_{\beta,p}(x) := \beta x^p-x\log x$.
\begin{defn}
    \label{point_define}
    We consider the following $3$-component partition of the parameter space:
    \begin{enumerate}
        \item A point $(\beta,h) \in \Theta$ is called \textit{regular}, if the function $H_{\beta,h}$ has a unique global maximizer $\bm m_*$ and the quadratic form
              $$\bm Q_{\bm v,\beta}(\bm t) := \sum_{r=1}^q \left(\beta p(p-1)v_r^{p-2} - \frac{1}{v_r}\right) t_r^2~,$$
              is negative definite on $\cH_q := \{\bm t\in \R^q: \sum_{r=1}^q t_r=0\}$ for $\bm v = \bm m_*$. The set of all regular points is denoted by $\cR_{p,q}. $
        \item A point $(\beta,h) \in \Theta$ is called \textit{critical}, if $H_{\beta,h}$ has more than one global maximizer, and for each such global maximizer $\bm m$, the quadratic form $\bm Q_{\bm m,\beta}$ is negative definite on $\cH_q$. The set of all critical points is denoted by $\cC_{p,q}$.
        \item A point $(\beta,h) \in \Theta$ is called \textit{special}, if $H_{\beta,h}$ has a unique global maximizer $\bm m_*$ and the quadratic form $\bm Q_{\bm m_*,\beta}$ is singular on $\cH_q$ (i.e. $\mathrm{Ker}(\bm Q_{\bm m_*,\beta}) \bigcap \cH_q \ne \{\boldsymbol{0}\}$). The set of all special points is denoted by $\cS_{p,q}$.
    \end{enumerate}
\end{defn}

\begin{defn}
    We can further classify the special points into the following two categories:
    \begin{enumerate}[i.]
        \item A special point $(\beta,h)$ is said to be of \textit{type-I}, if the unique global maximizer $\bm m_* =: \bm x_s$ satisfies $f_{\beta,h}^{(4)}(s)<0$. The set of all type-I special points is denoted by $\cS^1_{p,q}$.
        \item A special point $(\beta,h)$ is said to be of \textit{type-II}, if the unique global maximizer $\bm m_* =: \bm x_s$ satisfies $f_{\beta,h}^{(4)}(s)=0$. We denote the set of all type-II special points by $\cS^2_{p,q}$.
    \end{enumerate}
\end{defn}

It was shown in \cite{bhowal_mukh} that at all regular parameter points, the magnetization vector $\cs$ is asymptotically normal with mean being the maximizer of $H_{\beta,h}$, and the same is true at all critical points too, conditional on the event that $\cs$ lies in a small neighborhood around one of the maximizers whose closure excludes all other maximizers. The convergence of the magnetization vector to the corresponding maximizer also happens at the classical parametric rate $N^{-1/2}$ at regular and critical points. The story is more delicate at the special points. There, the rate of convergence of $\cs$ to the maximizer $\bm m_*$ is slower than the typical $N^{-1/2}$ speed, which is either $N^{-1/4}$ at the type-I special points or $N^{-1/6}$ at the type-II special points. Moreover, the limiting distribution at these special points is non-Gaussian. In particular, it is a constant vector multiple of a generalized normal distribution with shape parameter $4$ or $6$, depending on whether the special parameter point is of type-I or II, respectively. In this paper, we are going to derive Berry-Esseen type bounds for these distributional convergences, using the method of exchangeable pairs. Some preliminaries required for this approach are highlighted in the next section.

\subsection{The Method of Exchangeable Pairs}\label{mep}
In this section, we introduce some basics about the method of exchangeable pairs, which is the crucial tool for our analysis. We begin with the following fundamental definition:

\begin{defn}
    A pair of $\R^d$-values random vectors, $\left(\bm W, \bm W'\right)$ is called an exchangeable pair if the joint distribution of $(\bm W', \bm W)$ is same as the joint distribution of $(\bm W, \bm W')$.
\end{defn}
The following condition will be crucial in using the Stein's method for our analysis.
\begin{defn}
    We say that an exchangeable pair $\left(\bm W, \bm W'\right)$ of $\mathbb{R}^d$-valued random vectors satisfies the \textit{approximate linear regression condition} with remainder term $R(\bm W)$, if
    \begin{equation}
        \label{lin_reg}
        \E\left[\bm W'-\bm W \mid \bm W\right]=-\Lambda \bm W+R(\bm W),
    \end{equation}
    for an invertible matrix $\Lambda$.
\end{defn}

We now discuss the different Stein's method results that would act as necessary tools for our analysis.
The first theorem is taken from \cite{reinert2009multivariate}, which we state here for the sake of completion.
\begin{thm}
    \label{stein_third}
    Let $\left(\bm W, \bm W'\right)$ be an exchangeable pair of $d$-dimenional random vectors such that
    $$
        \E[\bm W]=0, \quad \E\left[\bm W \bm W^\top\right]=\Sigma,
    $$
    with $\Sigma \in \mathbb{R}^{d \times d}$ symmetric and positive definite. Assume that $(\bm W, \bm W')$ satisfies the approximate linear regression condition (defined in \eqref{lin_reg}) with remainder term $R$. Let $\bm Z$ be a $d$-dimensional standard normal vector. Then we have for every thrice differentiable function $g$,
    $$
        \left|\E g(\bm W)-\E g\left(\Sigma^{1 / 2} \bm Z\right)\right| \leq \frac{|g|_2}{4} A+\frac{|g|_3}{12} B+\left(|g|_1+\frac{1}{2} d\|\Sigma\|^{1 / 2}|g|_2\right) C,
    $$
    where
    $\lambda^{(i)}:=\sum_{m=1}^d\left|\left(\Lambda^{-1}\right)_{m, i}\right|$, $|g|_m := \sup_{i_1,\ldots,i_m} \left\|\frac{\partial^m}{\partial x_1\ldots \partial x_m} g\right\|$, and

    \begin{equation}
        \label{abc}
        \begin{aligned}
            A & =\sum_{i, j=1}^d \lambda^{(i)} \sqrt{\Var\left[\E\left((W_i'-W_i)(W_j'-W_j) \mid W\right)\right]}                \\
            B & =\sum_{i, j, k=1}^d \lambda^{(i)} \E\left|\left(W_i'-W_i\right)\left(W_j'-W_j\right)\left(W_k'-W_k\right)\right| \\
            C & =\sum_{i=1}^d \lambda^{(i)} \sqrt{\E (R_i^2)}.
        \end{aligned}
    \end{equation}
\end{thm}

The above theorem can be generalized to obtain a uniform bound on a class of more general functions. To begin with, we introduce a few notations. For any function $g: \mathbb{R}^d \rightarrow \mathbb{R}$, let us define:
$$
    \begin{aligned}
        g_\delta^{+}(x)      & =\sup \{g(x+y):|y| \leq \delta\}, \\
        g_\delta^{-}(x)      & =\inf \{g(x+y):|y| \leq \delta\}, \\
        \tilde{g}(x, \delta) & =g_\delta^{+}(x)-g_\delta^{-}(x).
    \end{aligned}
$$
Let us now define a new class $\mathcal{G}$ of real measurable functions on $\mathbb{R}^d$ (following \cite{eichelsbacher2015rates}) as follows:
\begin{enumerate}
    \item The functions $g \in \mathcal{G}$ are uniformly bounded.
    \item The class $\mathcal{G}$ is affinely invariant, i.e. for any $q \times q$ matrix $A$ and any vector $b \in \mathbb{R}^q$, the function $x\mapsto g(A x+b) \in \mathcal{G}$.
    \item For any $\delta>0$ and any $g \in \mathcal{G}$, the functions, $ g_\delta^{+}$ and $g_\delta^{-}$ are in $\mathcal{G}$.
    \item There exists a constant $c=c(\mathcal{G}, q)$ such that $ \sup _{g \in \mathcal{G}}\left\{\int_{\mathbb{R}^q} \tilde{g}(x, \delta) \phi(x)~d x\right\} \leq c \delta$, where $\phi$ denotes the $q$-dimensional standard normal density.
\end{enumerate}

The next theorem (Theorem 2.2 in \cite{eichelsbacher2015rates}) generalizes Theorem \ref{stein_third} to obtain a uniform bound for the difference between $\e g(\bm W)$ and $\e g(\Sigma^{1/2} \bm Z)$ for all $g$ over the class $\cG$.
\begin{thm}
    \label{stein_general}
    Let $\left(\bm W, \bm W'\right)$ be an exchangeable pair of $d$-dimenional random vectors such that
    $$
        \E[\bm W]=0, \quad \E\left[\bm W \bm W^T\right]=\Sigma,
    $$
    with $\Sigma \in \mathbb{R}^{d \times d}$ symmetric and positive definite. Assume that $(\bm W,\bm W')$ satisfies the approximate linear regression condition (defined in \eqref{lin_reg}) with remainder term $R$. Moreover, assume for $\left|W_i'-W_i\right| \leq A$ for every $i$. Let $\bm Z$ be a $d$-dimensional standard normal vector. Then,
    $$
        \begin{aligned}
            \sup _{g \in \mathcal{G}}\left|\E g(\bm W)-\E g\left(\Sigma^{1 / 2} \bm Z\right)\right| \leq & C\left[\log \left(t^{-1}\right) A_1+\left(\log \left(t^{-1}\right)\|\Sigma\|^{1 / 2}+1\right) A_2\right. \\
                                                                                                         & \left.+\left(1+\log \left(t^{-1}\right) \sum_{i=1}^d \E\left|W_i\right|+c\right) A^3 A_3+c A\right]
        \end{aligned}
    $$
    where

    \begin{equation*}
        \begin{aligned}
             & A_1=\sum_{i, j=1}^d\left|\left(\Lambda^{-1}\right)_{j, i}\right| \sqrt{\Var\left[\E\left((W_i'-W_i)(W_j'-W_j) \mid W\right)\right]},                                                             \\
             & A_2=\sum_{i, j=1}^d\left|\left(\Lambda^{-1}\right)_{j, i}\right| \sqrt{\E\left[R_i^2\right]}, \quad A_3=\sum_{i=1}^d \max _{j \in\{1, \ldots, d\}}\left|\left(\Lambda^{-1}\right)_{j, i}\right|,
        \end{aligned}
    \end{equation*}
    where $C$ denotes a constant that depends on $d, \sqrt{t}=2 C A^3 A_3$ and $c>1$ is taken from condition (4) on $\mathcal{G}$.
\end{thm}

Clearly, Theorems \ref{stein_third} and \ref{stein_general} cannot be applied if $\bm W$ has non-Gaussian asymptotics, which is true at the special points. For handling this case, we need a method from \cite{eichelsbacher2015rates} that will help us provide Berry-Esseen type bounds for non-normal approximation. Before introducing this method, we need a couple of definitions.

\begin{defn}
    We will call a function $f$ \textit{steady} on an interval $I=[a, b],-\infty \leq a<b \leq \infty$ (in \cite{stein2004use} and \cite{eichelsbacher2015rates}, this kind of functions \textit{regular}, but to avoid confusion with regular points we introduce this change of terminology) if:
    \begin{enumerate}
        \item  $f$ is finite on $I$,
        \item    at any interior point in $I$, $f$ possesses a right-hand limit and a left-hand limit, and
        \item the left-hand limit $f(b-)$ and the right-hand limit $f(a+)$ exist.
    \end{enumerate}
\end{defn}
\begin{defn}
    A steady, strictly positive density $p$ on an interval $I := [a,b]$ is said to be \textit{nice} (following the terminology of \cite{eichelsbacher2015rates}), if:
    \begin{enumerate}
        \item the derivative $p'$ exists on $I$ with countably many sign changes,
        \item $p'$ is continuous at the sign changes,
        \item  $\int_I p|\log p| <\infty$, and

        \item The function $\psi(x) := p'(x)/p(x)$ is steady on $I$.
    \end{enumerate}
\end{defn}

Proposition 1.4 in \cite{stein2004use} shows that a random variable $Y$ is distributed according to the density $p$ if and only if
\begin{equation*}
    \E\left[f'(Y)+\psi(Y) f(Y)\right]= f(b-) p(b-)-f(a+) p(a+)
\end{equation*}
for all steady functions $f$ on $I$ possessing (piecewise) steady derivative $f'$ on $I$, satisfying:

\begin{align}
    \label{suit_class_1}
     & \int_I p(x)|f'(x)|dx<\infty         \\
    \label{suit_class_2}
     & \int_I p(x)|f(x)\psi(x)|dx < \infty
\end{align}
The corresponding Stein identity is
$$
    f'(x)+\psi(x) f(x)=g(x)-P(g),
$$
where $g$ is a measurable function for which $\int_I|g(x)| p(x) d x<\infty$, $P(x):=\int_{-\infty}^x p(y) d y$ and $P(g):=\int_I g(y) p(y) d y$. We now restate Theorem 2.3 in \cite{eichelsbacher2015rates}. 
\begin{thm}
    \label{stein_special}
    Let $\left(W, W'\right)$ be an exchangeable pair of real-valued random variables, satisfying
    $$
        \E\left[W'-W \mid W\right]=\lambda \psi(W)-R(W)
    $$
    for some random variable $R=R(W)$, $0<\lambda<1$ and $\psi := p'/p$ with $p$ being a nice density. Let $p_W$ be a density such that a random variable $Z_W$ is distributed according to $p_W$ if and only if
    \begin{equation}\label{2p14}
        \E\left(\E[W \psi(W)] f'\left(Z_W\right)-\psi\left(Z_W\right) f\left(Z_W\right)\right)=0
    \end{equation}
    for all steady functions $f$ with steady derivatives, satisfying \eqref{suit_class_1} and \eqref{suit_class_2}.
    \begin{enumerate}
        \item Let us assume that for any absolutely continuous function $g$, the solution $f_g$ of \eqref{2p14} satisfies:
              $$
                  \left\|f_g\right\| \leq c_1\left\|g'\right\|, \quad\left\|f_g'\right\| \leq c_2\left\|g'\right\| \quad \text { and } \quad\left\|f_g^{\prime \prime}\right\| \leq c_3\left\|g'\right\|,
              $$
    \end{enumerate}
    Then for any uniformly Lipschitz function $g$, we have $\left|\E[g(W)]-\E\left[g\left(Z_W\right)\right]\right| \leq \delta\left\|g'\right\|$ with
    $$
        \delta:=\frac{c_2}{2 \lambda}\left(\Var\left(\E\left[\left(W-W'\right)^2 \mid W\right]\right)\right)^{1 / 2}+\frac{c_3}{4 \lambda} \E\left|W-W'\right|^3+\frac{c_1+c_2 \sqrt{\E W^2}}{\lambda} \sqrt{\E R^2}.
    $$
    \begin{enumerate}[(2)]
        \item Let us assume that for any function $g(x):=1_{\{x \leq z\}}(x), z \in \mathbb{R}$, the solution $f_z$ of \eqref{2p14} satisfies
              $$
                  \left|f_z(x)\right| \leq d_1, \quad\left|f_z'(x)\right| \leq d_2, \left|f_z'(x)-f_z'(y)\right| \leq d_3,
              $$
              and
              $$
                  \left|\left(\psi(x) f_z(x)\right)'\right|=\left|\left(\frac{p'(x)}{p(x)} f_z(x)\right)'\right| \leq d_4
              $$
              for all real $x$ and $y$, where $d_1, d_2, d_3$ and $d_4$ are constants.
    \end{enumerate}
    Then for any $A>0$, we have:
    $$
        \begin{aligned}
            \sup _{t \in \mathbb{R}} ~ \Big| \p(W \leq t)-\int_{-\infty}^t p_W( t) d t \Big| \leq & ~\frac{d_2}{2 \lambda}\left(\Var\left(\E\left[\left(W'-W\right)^2 \mid W\right]\right)\right)^{1 / 2}                            \\
            +                                                                                     & \left(d_1 + d_2 \sqrt{\E W^2}+\frac{3}{2} A\right) \frac{\sqrt{\E R^2}}{\lambda}+\frac{1}{\lambda}\left(\frac{d_4 A^3}{4}\right) \\
            +                                                                                     & \frac{3 A}{2} \E |\psi(W)|+\frac{d_3}{2 \lambda} \E\left(\left(W-W'\right)^2 1_{\left\{\left|W-W'\right| \geq A\right\}}\right).
        \end{aligned}
    $$
\end{thm}

\section{Statements of Main Results}\label{sec:main_res}
In this section, we state our main results, i.e. bounds for the rates of distributional convergence of the (centered and scaled) magnetization vector to their asymptotic distributions. The results are presented separately for the three different cases depending on whether the true parameter point is regular, critical or special (see Definition \ref{point_define}).

\subsection{Regular Points}
We will start with the regular case, where we define:
\begin{equation*}
    \bm W = \bm W_N:=\sqrt{N}\left(\cs-\bm m_*\right).
\end{equation*}

\begin{thm}
    \label{regular_three}
    Assume that $(\beta, h) \in \cR_{p,q}$ and let $\bm m_*$ be the unique maximizer of $H_{\beta, h}$. If $\bm Z$ has the $q$-dimensional standard normal distribution, then there exists a constant $C>0$ (depending on $\beta, h,p,q$ and $g$), such that for every thrice differentiable function $g$,
    $$
        \left|\E g(\bm W_N)-\E g\left(\Sigma^{1 / 2} \bm Z\right)\right| \leq C  N^{-1 / 2},
    $$
    where $\Sigma:=\mathrm{Cov}(\bm W_N)$. Moreover, there exists a constant $D>0$ (depending on $\beta, h,p$ and $q$), such that:

    \begin{equation*}
        \sup_{g \in \mathcal{G}}\left|\E g(\bm W_N)-\E g\left(\Sigma^{1 / 2} \bm Z\right)\right| \leq  \frac{D\log N}{\sqrt{N}}~,
    \end{equation*}
    where the class $\mathcal{G}$ is defined in Section \ref{mep}.
\end{thm}

\begin{remark}
    It follows from the proof of Theorem \ref{regular_three} that its statement holds verbatim if $\bm W_N$ is replaced by the centered form $\sqrt{N}(\cs - \e \cs)$.
\end{remark}

\begin{cor}\label{actualBE}
    There exists a constant $D>0$ depending on $\beta,h,p$ and $q$, such that for all convex sets $U \subseteq \R^q$:
    $$\left|\p(\bm W_N \in U) - \p(\Sigma^{1/2}\bm Z  \in U)\right|\le \frac{D\log N}{\sqrt{N}}~.$$
\end{cor}

\subsection{Critical Points}
Let us now shift our attention to the case when the true parameter point is critical. In this case, we define:
$$\bm W^{(i)} = \bm W_N^{(i)}:=\sqrt{N}\left(\cs-\bm m_i\right)$$
where $\bm m_1,\ldots,\bm m_K$ are the $K$ global maximizers of $H_{\beta,h}$.

\begin{thm}\label{critical_three}
    Assume that $(\beta, h)\in \cC_{p,q}$ and let $\epsilon>0$ be smaller than the distance between any two global maximizers of $\hf$. If $\bm Z$ has the $q$-dimensional standard normal distribution, then there exists a constant $C>0$ (depending on $\beta, h, p,q$ and $g$), such that for every thrice differentiable function $g$, one has
    $$
        \left|\E_i g\left(\bm W_N^{(i)}\right)-\E g\left(\Sigma_i^{1 / 2} Z\right)\right| \leq C  N^{-1 / 2},
    $$
    under the conditional probability measure
    \begin{equation}\label{condpcrc}
        \p_i := \p_{\beta, h, N}\left(\cdot \vert\, \cs \in B\left(\bm m_i, \epsilon\right)\right),
    \end{equation}
    where $\Sigma_i:=\mathrm{Cov}\left(\bm W_N^{(i)}\right)$, and $\E_i$ denotes the expectation with respect to the conditional measure \eqref{condpcrc}. Moreover, there exists a constant $D>0$ (depending on $\beta, h,p$ and $q$), such that:

    \begin{equation*}
        \sup_{g \in \mathcal{G}}\left|\E_i g\left(\bm W_N^{(i)}\right)-\E g\left(\Sigma_i^{1 / 2} \bm Z\right)\right| \leq  \frac{D\log N}{\sqrt{N}}~.
    \end{equation*}
\end{thm}

\begin{cor}\label{actualBEcr}
    There exists a constant $D>0$ depending on $\beta,h,p$ and $q$, such that for all convex sets $U \in R^q$:
    $$\left|\p_i(\bm W_N  \in U) - \p(\Sigma_i^{1/2}\bm Z  \in U) \right|\le \frac{D\log N}{\sqrt{N}}~.$$
\end{cor}

The proofs of Theorem \ref{critical_three} and Corollary \ref{actualBEcr} are so similar to the proofs of Theorem \ref{regular_three} and Corollary \ref{actualBE} (modulo working under the conditional distribution $\p_i$), that we ignore them.

\subsection{Special Points}
Finally, we consider the case when the true parameter point is special. Note that we can decompose $\cs - \bm m_*$ uniquely as:
\begin{equation*}
    \cs - \bm m_*=N^{-\frac{1}{4}}T_N \bm u+ N^{-\frac{1}{2}}\bm V_N,
\end{equation*}
for some $T_N$ and $\bm V_N \in \cH_q \cap\;\operatorname{Span}(\bm u)^\perp$, where $\bm u := (1-q,1,\ldots,1)\in \mathbb{R}^q$ and recall that $\cH_q$ is the simplex of all $q$-dimensional zero-sum vectors. Two different cases can arise depending on whether the true parameter is type-I or type-II special.

\begin{thm}
    \label{special_one}
    Assume that $(\beta, h) \in \cS^1_{p,q}$ and let $\bm m_*=\bm x_s$ be the unique maximizer of $H_{\beta, h}$. Furthermore, let $Z_{T_N}$ be a random variable distributed according to the probability measure with density proportional to
    $$
        \exp\left(-\frac{x^4}{4\E T_N^4}\right),
    $$
    Then, for any uniformly 1-Lipschitz function $g: \mathbb{R} \rightarrow \mathbb{R}$, we have:
    $$
        \left|\E g(T_N)-\E g\left(Z_{T_N}\right)\right| \leq \frac{C}{N^{1 / 4}}
    $$
    and
    $$
        \sup _{t \in \mathbb{R}}\left|\mathbb{P}(T_N \leq t)-\Phi_{q, T_N}(t)\right| \leq \frac{C}{N^{1 / 4}},
    $$
    where $\Phi_{q, T_N}$ denotes the distribution function of $T_N$, and the constant $C>0$ depends on $\beta,h,p,q$.

    \noindent Moreover, for $\Sigma:=\E\left[\bm V_N \bm V_N^T\right]$, we have for every thrice differentiable function $g$,
    \begin{equation*}
        \left|\E g(\bm V_N)-\E g\left(\Sigma^{1 / 2} \bm Z\right)\right| \leq \frac{C}{N^{1 / 4}},
    \end{equation*}
    where $C>0$ is a constant depending on $\beta,h,p,q,g$ and $\bm Z$ has the $q$-dimensional standard normal distribution.  Moreover, there exists a constant $D>0$ (depending on $\beta, h,p$ and $q$), such that:

    \begin{equation*}
        \sup_{g \in \mathcal{G}}\left|\E g\left(\bm V_N\right)-\E g\left(\Sigma^{1 / 2} \bm Z\right)\right| \leq  \frac{D\log N}{N^{1/4}}~.
    \end{equation*}
\end{thm}

Now, a type-II special point arises only when $(p,q)=(4,2)$, in which case, $\bm u = (-1,1)$ (see Lemma F.2 in \cite{bhowal_mukh}). In this case, we can uniquely write:
\begin{equation*}
    \cs - \bm m_*=N^{-\frac{1}{6}}F_N \bm u.
\end{equation*}
\begin{thm}
    \label{special_two}
    Assume that $(\beta, h) \in \cS^2_{p,q}$ and let $Z_{F_N}$ be a random variable distributed according to the probability measure with density proportional to
    $$
        \exp\left(-\frac{x^6}{6 \E F_N^6}\right).
    $$
    Then, for any uniformly 1-Lipschitz function $g: \mathbb{R} \rightarrow \mathbb{R}$, we have:
    $$
        \left|\E g(F_N)-\E g\left(Z_{F_N}\right)\right| \leq \frac{C}{N^{1 / 6}}
    $$
    and
    $$
        \sup _{t \in \mathbb{R}}\left|\mathbb{P}(F_N \leq t)-\Phi_{q, F_N}(t)\right| \leq \frac{C}{N^{1 / 6}},
    $$
    where $\Phi_{q, F}$ denotes the distribution function of $F$, and the constant $C>0$ depends on $\beta,h,p$ and $q$.
\end{thm}

The proofs of these main results are given in the next section. The central idea behind the proofs is to use Stein's method of exchangeable pairs, by first sampling $\bm X$ from \eqref{eq:cp}, and then constructing a new random vector $\bm X' \in [q]^N$ by replacing a randomly chosen entry of $\bm X$ with a random variate generated from the conditional distribution of that entry given the others. For regular and critical points, this step is then followed by expressing the conditional expectation $\e[\bm W-\bm W'|\bm W]$ as a linear map of $\bm W$ plus a certain remainder term, where $\bm W'$ is defined from $\bm X'$ in the same way $\bm W$ is defined from $\bm X$. The rest of the analysis then circles around giving appropriate bounds on the remainder terms, which is necessary for applying Stein's method to obtain the final rate of distributional convergence. A similar argument holds for special points too, the only difference being that the main term in the conditional expectation $\e[T-T'|T]$ (or $E[F-F'|F]$) is now a non-linear function of $T$ (or $F$), more precisely the function $\psi(x) = x^3$ (or $x^5$).


\subsection{Berry-Esseen Bounds for the Maximum Pseudolikelihood Estimate of $\beta$}\label{sec:estimationpl}

Here, we consider the problem of estimating the parameter $\beta$ (assuming $h=0$) given a single sample $\bm X = (X_1, X_2, \ldots, X_n)$ from the model \eqref{eq:cp}. Maximum likelihood estimation in this model might be computationally expensive due to the presence of the complicated normalizing constant $Z_N$. An alternative, more computationally feasible approach is the \textit{maximum pseudolikelihood} (MPL) method, which uses the fact that the conditional distributions $\p_{\beta,0,p}(X_i|(X_j)_{j\ne i})$ are explicit. With this in mind, the MPL estimate of $\beta$ in the model \eqref{eq:cp} (with $h=0$) can be defined as:
$$\hat{\beta} := \arg \max_{\beta\in \mathbb{R}} L(\beta|\bm X) , $$
where
$$L(\beta|\bm X) := \prod_{i=1}^N \p_{\beta,0,p}\left(X_i|(X_j)_{j\ne i}\right)$$ is known as the \textit{pseudolikelihood} function.

It follows from Lemma F.7 in \cite{bhowal_mukh} that there exists a threshold $\beta_c = \beta_c(p,q) >0$ such that for $\beta>\beta_c$, the function $\hf$ has exactly $q$ maximizers $\bm m_1,\ldots,\bm m_q$, which are permutations of one another. In the next theorem, we derive Berry-Esseen bound for weak convergence of the MPL estimate $\hat{\beta}$ over the region $\beta > \beta_c$. Berry-Esseen bounds of the MPL estimator in the closely related tensor Curie-Weiss Ising model was recently derived in \cite{smstein}.

\begin{thm}\label{thm:bempl}
    Suppose $\hat \beta$ is the MPL estimate of $\beta$ in the model $\p_{\beta,0,N}$. Then for $\beta > \beta_c$,
    \begin{equation*}
        \sup_{x\in \mathbb{R}}~\left|\p\left(\sqrt{N}(\hat{\beta}-\beta) \leq x\right) - \frac{1}{q}\sum_{i=1}^{q}\p\left(\boldsymbol{\rho}_i^T\Sigma_i^{1/2} \bm Z \leq x\right)\right| = O\left(\frac{\log N}{\sqrt{N}}\right) ,
    \end{equation*}
    where $Z$ is a standard Normal random variable and,
    \begin{equation}
        \boldsymbol{\rho}_i=-\frac{\nabla_{\bm x}  S(\bm m_i,\beta)}{\frac{\partial}{\partial \beta} S(\bm m_i,\beta)},
    \end{equation}
    where
    \begin{equation*}
        S(\bm x,\beta) := ||\bm x||_p^p-\frac{\sum_{r=1}^{q}x_r^{p-1}\exp(\beta p x_r^{p-1})}{\sum_{r=1}^{q}\exp(\beta p x_r^{p-1})}~.
    \end{equation*}
\end{thm}

\section{Proofs of Theorems}\label{sec:proof}
\begin{proof}[Proof of Theorem \ref{regular_three}]
    Let us first consider the case $h=0$. Define a new random vector $\bm X' \in [q]^N$ by first picking $I$ uniformly from the set $\{1,\cdots,N\}$ independent of all other random variables involved, and if $I=j$, then just replacing $X_j$ by $X'_j$ drawn from the conditional distribution of $X_j$ given $(X_t)_{t\neq j}$. Define $\bm X_j := (X_{j,1},\cdots,X_{j,q}) = (\one_{X_j=1},\cdots,\one_{X_j=q})$, $ \bm X'_j := (X'_{j,1},\cdots,X'_{j,q}) = (\one_{X'_j=1},\cdots,\one_{X'_j=q})$ and
    \begin{equation*}
        \bm W' := \bm W - \frac{\bm X_I}{\sqrt{N}} + \frac{\bm X_I'}{\sqrt{N}}.
    \end{equation*}
    Now, write $x_0=(x_{0,1},\ldots, x_{0,q})$ (recall the notation from \eqref{xsdefine}) and define:
    $$m_{i,r} := \frac{1}{N}\sum_{t\ne i} X_{t,r}~.$$
    In these notations, we have for each $r\in [q]$,
    $$
        \begin{aligned}
            \E\left[W_{r}'-W_{r} \Big|\left\{X_{i}\right\}_{i=1}^{N}\right] & =\frac{1}{\sqrt{N}} \frac{1}{N} \sum_{j=1}^{N} \E\left[X_{j, r}'-X_{j, r} \Big|\left\{X_{i}\right\}_{i=1}^{N}\right]
            =-\frac{1}{N} \frac{1}{\sqrt{N}} \sum_{j=1}^{N} \left(X_{j, r}- \frac{\exp \left(p \beta m_{j, r}^{p-1}\right)}{\sum_{s=1}^{q} \exp \left(p \beta m_{j, s}^{p-1}\right)} \right)                                                                                                 \\
                                                                            & =-\frac{1}{\sqrt{N}}\left(\frac{W_{r}}{\sqrt{N}}+x_{0, r}\right)+ \frac{1}{N \sqrt{N}} \sum_{j=1}^{N} \frac{\exp \left(p \beta m_{j, r}^{p-1}\right)}{\sum_{s=1}^{q} \exp \left(p \beta m_{j, s}^{p-1}\right)} \\
                                                                            & =-\frac{W_r}{N}-\frac{x_{0, r}}{\sqrt{N}}+\frac{1}{\sqrt{N}}\left(\frac{-1}{\beta p(p-1)\Xrb^{p-2}} \nabla_r G_{\beta, h}\left(\cs\right)+\Xrb\right)+R_{N}^{(1)}(r)                                           \\
                                                                            & =-\frac{1}{\sqrt{N}} \frac{1}{\beta p(p-1) \Xrb^{p-2}} \nabla_r G_{\beta, h}\left(\cs\right)+R_{N}^{(1)}(r),                                                                                                   \\
        \end{aligned}
    $$
    where $R_{N}^{(1)}(r):=N^{-3/2} \sum_{j=1}^{N}\left(\frac{\exp \left(p \beta m_{j, r}^{p-1}\right)}{\sum_{s=1}^{q} \exp \left(p \beta m_{j, s}^{p-1}\right)}-\frac{\exp \left(p \beta \Xrb^{p-1}\right)}{\sum_{s=1}^{q} \exp \left(p \beta \Xsb^{p-1}\right)}\right)$.

    Now, Taylor expanding the last step, we get:
    \begin{equation}
        \label{reg_stein}
        \E\left[W_{r}'-W_{r} \Big|\left\{X_{i}\right\}_{i=1}^{N}\right] =-\frac{1}{\beta p(p-1)} \frac{1}{N}\left\langle\Lambda_{r}, \bm W\right\rangle+R_{N}^{(1)}(r)+R_{N}^{(2)}(r),
    \end{equation}
    where $\Lambda_r$ denotes the gradient of the function
    $G_r(\bm x) := x_r^{2-p} \nabla_r G_{\beta,h}(\bm x)$ at the point $\bm x_0$ and
    $$R_{N}^{(2)}(r) := -\frac{1}{N\sqrt{N}}\frac{1}{\beta p (p-1)} \bm W^\top \nabla^2 G_r(\xi) \bm W$$ for some point $\xi$ lying on the line segment joining $\cs$ and $\bm x_0$. It follows from \cite{bhowal_mukh} that the quantity $\e[\bm W^\top \nabla^2 G_r(\xi) \bm W]^\ell$ is $O(1)$ for all positive integers $\ell$ and hence, $\e R_{N}^{(2)}(r)^\ell=O\left(N^{-3 \ell/ 2}\right)$. Next, towards bounding $R_n^{(1)}(r)$, note that:

    \begin{equation}
        \label{rn1_bound}
        \begin{aligned}
            \lvert R_{N}^{(1)}(r)\rvert & =N^{-3/2} \sum_{j=1}^{N}\left\lvert \frac{\exp \left(\beta p m_{j, r}^{p-1}\right)}{\sum_{s=1}^{q} \exp \left(\beta p m_{j, s}^{p-1}\right)}-\frac{\exp \left(\beta p \Xrb^{p-1}\right)}{\sum_{s=1}^{q} \exp \left(\beta p\Xsb^{p-1}\right)}\right\rvert \\
                                        & \leq N^{-3/2} q^{-2}\sum_{j=1}^N \sum_{s = 1}^q\left|\exp \left(\beta p m_{j, r}^{p-1}+\beta p \Xsb^{p-1}\right)-\exp \left(\beta p \Xrb^{p-1}+\beta p m_{j, s}^{p-1}\right)\right|                                                                      \\
                                        & \leq  N^{-3/2} q^{-2}\beta p e^{2 \beta p} \sum_{j=1}^{N} \sum_{s=1}^q\left|m_{j,r}^{p-1} - m_{j,s}^{p-1} + \Xsb^{p-1} - \Xrb^{p-1}\right|                                                                                                               \\
                                        & \leq N^{-3/2} q^{-2} \beta p e^{2 \beta p} \sum_{j=1}^{N} \sum_{s=1}^q \left\{\left|\left(\Xrb-\frac{X_{j,r}}{N}\right)^{p-1}-\Xrb^{p-1}\right|+\left|\left(\Xsb-\frac{X_{j,s}}{N}\right)^{p-1}-\Xsb^{p-1}\right|\right\}                                \\
                                        & \leq 2 N^{-3/2} q^{-1}(p-1) p \beta e^{2 p \beta},
        \end{aligned}
    \end{equation}
    where the third inequality follows from the fact:
    $$
        |\exp (\alpha x)-\exp (\alpha y)| \leq \frac{|\alpha|}{2}(\exp (\alpha x)+\exp (\alpha y))|x-y| \text {, for all } \alpha, x, y \in \mathbb{R}.
    $$ as in \cite{eichelsbacher2015rates}.
    Hence, $R_{N}^{(1)}(r)=O\left(N^{-3 / 2}\right)$.

    Now, it follows from Lemma \ref{proplamb} that the matrix $\Lambda := [\Lambda_1,\cdots,\Lambda_q]^\top$ is invertible, and hence by \eqref{reg_stein}, $(\bm W, \bm W')$ satisfies the approximate linear regression condition with remainder term $R_N^{(1)} + R_N^{(2)}$. However, $\E \bm W \neq 0$. So, we will apply Theorem \ref{stein_third} to the pair $(\bm W_1, \bm W_1') := (\bm W - \E \bm W, \bm W' - \E \bm W')$.  Hence, from \eqref{reg_stein} we get that,
    \begin{equation*}
        \E\left[\bm W'_1- \bm W_1 \Big|\left\{X_{i}\right\}_{i=1}^{N}\right] =-\frac{1}{\beta p(p-1)} \frac{1}{N} \Lambda \bm W_1 + R_N,
    \end{equation*}
    where $R_N=-\frac{1}{\beta p(p-1)} \frac{1}{N} \Lambda \E \bm W +R_{N}^{(1)}+R_{N}^{(2)}$. From Lemma \ref{reg_expec_bd}, we see that $\E \bm W =  O(N^{-\frac{1}{2}})$. Hence, $\e \|R_N\|_\ell^\ell = O(N^{-\frac{3\ell}{2}})$ for all positive integers $\ell$.

    The bounds on A, B, C as in Theorem \ref{stein_third} depend only on $\bm W'- \bm W$ which is same as $\bm W_1 - \bm W'_1$.
    Since the approximate linear regression condition is satisfied with coefficient matrix equal to $(\beta p (p-1))^{-1} N^{-1} \Lambda$, we clearly have $\lambda^{(r)} = O(N)$ for all $r$.
    Now, we will try to bound $A,B,C$ as defined in \eqref{abc}. Towards this, we have:
    \begin{equation*}
        C=\sum_{r=1}^{q}\lambda^{(r)}\sqrt{\E[R_r^2]}=O\left(N\cdot N^{-3 / 2}\right) = O(N^{-1/2}).
    \end{equation*}
    Also, $|W_{r}-W_{r}'| = O(1/\sqrt{N})$ which implies that $B=O\left(N^{-1 / 2}\right)$. Finally, towards bounding $A$, let $\cF$ be the $\sigma$-field generated by $\bm X_1,\ldots,\bm X_N$. Then, we have:
    \begin{equation*}
        \begin{aligned}
             & \E\left[(W_{r}'-W_{r})(W_s'-W_s) \mid \cF\right]=\frac{1}{N^{3}} \sum_{k, t=1}^{N} X_{k, r} X_{t, s}+\frac{1}{N^{3}} \sum_{k, t=1}^{N} \E\left[X_{k, r}' X_{t, s}' \mid \cF\right]-\frac{2}{N^{3}} \sum_{k, t=1}^{N} X_{k, r} \E\left[X_{t, s}' \mid \cF\right] \\
             & =: A_{1,r,s}+A_{2,r,s}+A_{3,r,s}.
        \end{aligned}
    \end{equation*}
    Let us first upper bound $\mathrm{Var}[A_{1,r,s}]$.
    \begin{equation*}
        \begin{aligned}
            \Var[A_{1,r,s}] & =\Var\left[\frac{1}{N^{3}} \sum_{k, t=1}^{N} X_{k, r} X_{t, s}\right]=\frac{1}{N^{2}} \Var\left[\Xrb \Xsb\right]                                                \\
                            & =\frac{1}{N^{2}} \Var\left[\frac{W_{r} W_{s}}{N}+\frac{W_{r}}{\sqrt{N}} x_{0,s}+\frac{W_{s}}{\sqrt{N}} x_{0, r}\right]                                          \\
                            & \lesssim \frac{1}{N^{2}} \cdot \max \left(\frac{1}{N^{2}} \Var\left(W_r W_{s}\right), \frac{1}{N} \Var\left(W_r\right), \frac{1}{N} \Var\left(W_s\right)\right) \\
                            & \lesssim \frac{1}{N^{4}}\left(\E\left[W_{r}^{2} W_{s}^{2}\right]+N \E\left[W_{r}^{2}\right]+N \E\left[W_{s}^{2}\right]\right)                                   \\
                            & =O\left(N^{-3}\right).
        \end{aligned}
    \end{equation*}
    Next, we upper bound $\mathrm{Var}[A_{2,r,s}]$.
    \begin{equation*}
        \begin{aligned}
            \Var\left[A_{2,r,s}\right]   \leq \Var\left[\frac{1}{N^{3}} \sum_{k, t=1}^{N} X_{k, r}' X_{t, s}'\right] = \Var[A_{1,r,s}]                                                                                                                                                                                        =O\left(N^{-3}\right).
        \end{aligned}
    \end{equation*}

    In order to bound $\mathrm{Var}[A_{3,r,s}]$, we will break $A_{3,r,s}$ down in the following way,
    \begin{equation*}
        \begin{aligned}
             & -\frac{A_{3,r,s}}{2}=\frac{1}{N^{3}} \sum_{t, k=1}^{N} X_{k, r} \frac{\exp \left(p \beta m_{t, s}^{p-1}\right)}{\sum_{r=1}^{q} \exp \left(p \beta m_{t,r}^{p-1}\right)}                                                                                          \\
             & =\frac{1}{N^{3}} \sum_{t, k=1}^{N} X_{k, r}\left(\frac{\exp \left(p \beta m_{t, s}^{p-1}\right)}{\sum_{r=1}^{q} \exp \left(p \beta m_{t,r}^{p-1}\right)}-\frac{\exp \left(p \beta \Xsb^{p-1}\right)}{\sum_{r=1}^{q} \exp \left(p \beta \Xrb^{p-1}\right)}\right) \\
             & +\frac{1}{N^{2}} \sum_{k=1}^{N} X_{k, r} \frac{\exp \left(p \beta \Xsb^{p-1}\right)}{\sum_{r=1}^{q} \exp \left(p \beta \Xrb^{p-1}\right)}                                                                                                                        \\
             & =: M_{1}+M_{2}.
        \end{aligned}
    \end{equation*}
    First, note that by \eqref{rn1_bound}, $|M_{1}| \leq C_{p,q,\beta} N^{-2}\Xrb =O_{p,q,\beta}(N^{-2})$ for some constant $C_{p,q,\beta}$ depending only on $p,q$ and $\beta$. This implies that $\Var[M_1] = O(N^{-4})$.


    Finally, let us estimate $M_2$,
    \begin{equation*}
        \begin{aligned}
            M_{2} & =\frac{1}{N} \Xrb\left(\Xsb-\frac{1}{\beta p (p-1)\Xsb^{p-2}} \nabla_s G_{\beta, 0}\left(\cs\right)\right)                               \\
                  & =\frac{1}{N} \Xrb\Xsb -\frac{1}{N} \Xrb \frac{1}{\beta p(p-1)}\left(\langle\Lambda_s,\cs - \bm x_0\rangle+\sqrt{N} R_{N}^{(2)}(r)\right) \\
                  & =\frac{1}{N} \Xrb\Xsb + O\left(N^{-3/2}\right) \Xrb \max_{r}W_{N,r}  +O\left(\frac{1}{\sqrt{N}} R_{N}^{(2)}(r)\right)                    \\
                  & =: S_1+S_2+S_3.
        \end{aligned}
    \end{equation*}

    Note that:
    \begin{eqnarray*}
        S_1 &=& \frac{1}{N^2} W_rW_s + \frac{1}{N} \left(\Xrb x_{0,s} + \Xsb x_{0,r}\right) - \frac{1}{N} x_{0,r}x_{0,s}\\ &=& \frac{1}{N^2} W_rW_s + \frac{1}{N^{3/2}} \left(W_rx_{0,s} + W_s x_{0,r}\right) + \frac{1}{N} x_{0,r}x_{0,s}
    \end{eqnarray*}
    and hence, $\Var(S_1) = O(N^{-3})$. Since $\Xrb \le 1$ and all moments of $|\max_r W_{N,r}|$ are bounded, We trivially have $\Var(S_2) = O(N^{-3})$. Finally, we had already established that $\e R_N^{(2)}(r)^2 = O(N^{-3})$ and hence, $\Var(S_3) = O(N^{-4})$. Therefore, $\Var(M_2) = O(N^{-3})$.
    Combining these together, we get $$\Var[A_{3,r,s}] = O\left(N^{-3}\right).$$ Therefore, we have:
    $$\Var\left[\E\left((W_{r}'-W_{r})(W_s'-W_s) \mid \cF\right)\right] = O(N^{-3})$$ implying that $A = O(N\cdot N^{-3/2}) = O(N^{-1/2})$.

    For $h>0$, \eqref{reg_stein} can be rewritten as
    \begin{equation*}
        \E\left[W_{r}'-W_{r} \mid\left\{X_{i}\right\}_{i=1}^{N}\right] =-\frac{1}{\beta p(p-1)} \frac{1}{N}\left\langle\Lambda_{r}, W\right\rangle+R_{N}^{(1)}(r,h)+R_{N}^{(2)}(r,h)
    \end{equation*}
    where $R_N^{(2)}(r,h)$ is defined in the same way same as before, and
    \begin{equation}\label{genremn1}
        R_N^{(1)}(r,h) := N^{-3/2} \sum_{j=1}^{N}\left(\frac{\exp \left(p \beta m_{j, r}^{p-1} + h\delta_{r,1}\right)}{\sum_{s=1}^{q} \exp \left(p \beta m_{j, s}^{p-1} + h\delta_{s,1}\right)}-\frac{\exp \left(p \beta \Xrb^{p-1} + h\delta_{r,1}\right)}{\sum_{s=1}^{q} \exp \left(p \beta \Xsb^{p-1} + h\delta_{s,1}\right)}\right)
    \end{equation}
    The bound for $R_{N}^{(2)}(r,h)$ remains same whereas that for $R_{N}^{(1)}(r,h)$ is multiplied by a factor of $e^h$. Also, by Lemma \ref{proplamb}, $\Lambda$ is invertible and hence Stein's method is applicable. The rest of the proof remains same. In conclusion, we have derived that:
    $$ \left|\E g\left(\bm W_N -\e\bm W_N\right)-\E g\left(\Sigma^{1 / 2} \bm Z\right)\right| \leq C_{\beta,h,q,g} N^{-1 / 2}~.$$
    Now, a one-term Taylor expansion of $g$ around $\bm W_N$ and Lemma \ref{reg_expec_bd} will give us the first part of Theorem \ref{regular_three}, in view of Theorem \ref{stein_third}.

    We now prove the second part of Theorem \ref{regular_three}.
    Majority of the proof follows similarly to Theorem \ref{regular_three}, allowing us to apply Theorem \ref{stein_general}. It remains to estimate the bound given there. Note that from previous calculations, we already have $A_1 = O(N^{-1/2})$ and $A_2=O(N^{-1/2})$. Also, trivially, $A= O(N^{-1/2})$ and $A_3 = O(N)$. Combining all these, we get:
    $$t= O(A^6 A_3^2) = O(N^{-1}).$$ Also, $A^3 A_3 = O(N^{-1/2})$. This shows that for some constant $D = D_{\beta,h,q}$, we have:
    $$ \sup_{g \in \mathcal{G}}\left|\E g(\bm W_N -\e \bm W_N)-\E g\left(\Sigma^{1 / 2} \bm Z\right)\right| \leq  \frac{D\log N}{\sqrt{N}}~.$$
    Now, define a new function $g_0(\bm x) := g(\bm x +\e \bm W_N)$. Then, we have:
    \begin{eqnarray}\label{finalmdf}
        &&\left|\E g(\bm W_N)-\E g\left(\Sigma^{1 / 2} \bm Z\right)\right|\nonumber\\ &=& \left|\e g_0 (\bm W_N -\e \bm W_N) - \e g_0\left(\Sigma^{1/2}\bm Z - \e\bm W_N\right)\right|\nonumber\\&\le&  \left|\e g_0 (\bm W_N -\e \bm W_N) - \e g_0\left(\Sigma^{1/2}\bm Z\right)\right| + \left|\e g_0 \left(\Sigma^{1/2}\bm Z\right) - \e g_0\left(\Sigma^{1/2}\bm Z - \e\bm W_N\right)\right|~.
    \end{eqnarray}
    Note that if $g\in \mathcal{G}$, then $g_0\in \mathcal{G}$, too. Hence, from what we have proved, the first term of \eqref{finalmdf} is bounded above by $D\log N/\sqrt{N}$. The second term of \eqref{finalmdf} can be written as:
    \begin{eqnarray*}
        &&\left|\int_{\mathbb{R}^q} \left(g_0(\Sigma^{1/2}\bm z - \e \bm W_N) - g_0(\Sigma^{1/2}\bm z)\right)~\phi(\bm z)~d\bm z\right|\\ &\le& \int_{\mathbb{R}^q} \tilde{g}_0(\Sigma^{1/2}\bm z,\e\bm W_N)~\phi(\bm z)~d\bm z = O_{\mathcal{G},q}(\e\bm W_N) = O_{\mathcal{G},q}(N^{-1/2})~.
    \end{eqnarray*}
    This completes the proof of the second part of Theorem \ref{regular_three}.
\end{proof}

\begin{proof}[Proof of Corollary \ref{actualBE}]
    Corollary \ref{actualBE} follows on taking the function class $\mathcal{G}$ to be the set of indicators on all convex subsets of $\mathbb{R}^q$.
\end{proof}

\begin{proof}[Proof of Theorem \ref{special_one}]
    To begin with, note that:
    \begin{equation*}
        T_N=\frac{N^{1 / 4}}{(1-q)}\left(\bar{X}_{\cdot 1}-m_1-\frac{V_{N,1}}{\sqrt{N}}\right)=\frac{1}{(1-q) N^{3/4}}\left(\sum_{i=1}^{N} X_{i, 1}-N m_{1}-\sqrt{N} V_{N,1}\right).
    \end{equation*}
    Note that the fact $\bm V_N \in \mathcal{H}_q\cap \mathrm{Span}(u)^\perp$ implies that $V_{N,1} = 0$.
    Let $\cF$ be the $\sigma$-field generated by $\bm X_1,\ldots,\bm X_N$. Just like the construction of $\bm W'$ in the proof for Theorem \ref{regular_three}, we construct $T_N'$, and note that:
    \begin{equation*}
        \begin{aligned}
            \E\left[T_N'-T_N \mid \cF\right] & =\frac{1}{(1-q)N^{7 / 4}} \sum_{i=1}^{N} \E\left[X_{i, 1}'-X_{i, 1} \mid \cF\right]                                   \\
                                             & =\frac{1}{(1-q)N^{7 / 4}} \sum_{i=1}^{N} \E\left[X_{i, 1}' \mid \cF\right]-\frac{T_N}{N} -\frac{m_1}{(1-q)N^{3 / 4}}.
        \end{aligned}
    \end{equation*}
    Now by \eqref{g_deriv}, we have:
    \begin{multline}
        \label{ey_special}
        \frac{1}{(1-q)N^{7 / 4}} \sum_{i=1}^{N} \E\left[X_{i, 1}' \mid \cF\right]=\frac{1}{(1-q) N^{3 / 4}}\left(\frac{-1}{\beta p(p-1) \bar{X}_{\cdot 1}^{p-2}} \nabla_1 G_{\beta, h}\left(\cs\right)+ \bar{X}_{\cdot 1}\right)\\
        +\frac{1}{(1-q) N^{1 / 4}} R_{N}^{(1)}
    \end{multline}
    where $R_N^{(1)}$ is as defined in \eqref{genremn1}.
    Using the arguments in \eqref{rn1_bound}, we get:
    \begin{equation*}
        \frac{1}{(1-q) N^{1 / 4}} R_{N}^{(1)}=O\left(\frac{1}{N^{7 / 4}} \right).
    \end{equation*}
    Also, by Lemma \ref{fourthTaylor},
    \begin{equation}
        \label{g_apply}
        \begin{aligned}
            \nabla_1\gf \left(\bm m_{*}+\frac{T_N u}{N^{1 / 4}}+\frac{V_N}{\sqrt{N}}\right) & =c(p,q)\frac{T_N^{3}}{N^{3 / 4}}+O\left(\sum_{j=2}^{q} \frac{V_{N,j}^{2}}{N}\right)+O\left(\frac{T_N^{4}}{N}\right)+O\left(\frac{T_N}{N^{1 / 4}} \sum_{j=2}^q \frac{V_{N,j}^{2}}{N}\right) \\
                                                                                            & =c(p,q)\frac{ T_N^{3}}{N^{3 / 4}}+R_N^{(2)},
        \end{aligned}
    \end{equation}
    where $\|R_N^{(2)}\|_\ell = O(\frac{1}{N})$ for all $\ell$.
    From \eqref{ey_special} and \eqref{g_apply}, we get that,
    \begin{equation*}
        \begin{aligned}
            \E\left[T_N'-T_N \mid \cF\right] & =- \frac{1}{(1-q) N^{3 / 4}}\cdot\frac{1}{\beta p(p-1) \bar{X}_{\cdot 1}^{p-2}}\left(\frac{c(p, q) T_N^{3}}{N^{3 / 4}}+R_n^{(2)}\right)+\frac{1}{(1-q) N^{1 / 4}} R_{N}^{(1)} \\
                                             & =- \frac{1}{(1-q) N^{3 / 4}}\cdot\frac{1}{\beta p(p-1) m_1^{p-2}}\left(\frac{c(p, q) T_N^{3}}{N^{3 / 4}}+R_n^{(2)}\right)+\frac{1}{(1-q) N^{1 / 4}} R_{N}^{(1)}+R_n^{(3)},
        \end{aligned}
    \end{equation*}
    where $R_N^{(3)}:=\frac{1}{(1-q) N^{3 / 4}}\left(\frac{c(p, q) T_N^{3}}{N^{3 / 4}}+R_n^{(2)}\right)\left(\frac{1}{\beta p(p-1) m_1^{p-2}}-\frac{1}{\beta p(p-1) \bar{X}_{\cdot 1}^{p-2}}\right)$.
    We thus have:
    $$\E\left[T_N'-T_N \mid \cF\right]=\lambda \psi(T_N)+R(T_N)$$
    where $\lambda :=\frac{1}{\beta p(p-1)(q-1)m_1^{p-2}N^{3 / 2}}$, $\psi(x)=c(p,q)x^3$, and
    \begin{equation*}
        R = R(T_N) := \frac{1}{(1-q) N^{3 / 4}}\frac{1}{\beta p(p-1) m_1^{p-2}}R_N^{(2)}+\frac{1}{(1-q) N^{1 / 4}} R_{N}^{(1)}+R_N^{(3)}.
    \end{equation*}

    We will use Theorem \ref{stein_special} to prove Theorem \ref{special_one}. By Lemma 2.2 in \cite{eichelsbacher2010stein}, all the required hypotheses of Theorem \ref{stein_special} are satisfied. Let us now bound $R_N^{(3)}$. Towards this, note that by a one-term Taylor expansion of the function $x^{2-p}$, one has:
    $$\frac{1}{\beta p(p-1) m_1^{p-2}}-\frac{1}{\beta p(p-1) \bar{X}_{\cdot 1}^{p-2}} = O\left(m_1-\bar{X}_{\cdot 1}\right)\implies \left\|\frac{1}{\beta p(p-1) m_1^{p-2}}-\frac{1}{\beta p(p-1) \bar{X}_{\cdot 1}^{p-2}}\right\|_\ell = N^{-1/4}$$ for all $\ell$. This, together with the fact that $R_N^{(2)} = O(1/N)$ implies that $\|R_N^{(3)}\|_\ell =  O(N^{-7/4})$ for all $\ell$.
    Combining the bounds on $R_N$ we get that $\|R(T_N)\|_\ell=O(N^{-7/4})$ for all $\ell$. Hence, we have:
    \begin{equation*}
        \frac{c_{1}+c_{2} \sqrt{\E\left[T_N^{2}\right]}}{\lambda} \sqrt{\E\left[R^{2}\right]}=O\left(N^{-1 / 4}\right)
    \end{equation*}
    for any constants $c_1$ and $c_2$.
    Next, taking $\bm W := \sqrt{N}(\cs - \bm m_*)$, we have:

    \begin{equation*}
        \E\left[\left(T_N'-T_N\right)^{2} \mid T_N\right]=\frac{1}{(1-q)^{2} N^{1 / 2}} \E\left[\left(W_{1}'-W_{1}\right)^{2} \mid W\right].
    \end{equation*}
    Again following the proof of Theorem \ref{regular_three} with the obvious modification that now, $\|W_1\|_\ell = O(N^{1/4})$ for all $\ell$, we see that $\mathrm{Var}\left( \E\left[\left(W_{1}'-W_{1}\right)^{2} \mid W\right]\right)=O\left(N^{-5 / 2}\right)$. Therefore,
    \begin{equation*}
        \frac{c_{2}}{\lambda} \left(\Var\E\left[\left(T_N'-T_N\right)^{2} \mid T_N\right]\right)^{1 / 2}=O\left(N^{-1 / 4}\right).
    \end{equation*}
    Since $T_N'-T_N = O(N^{-3/4})$, we have:
    \begin{equation*}
        \frac{c_{3}}{\lambda} \E\left|T_N-T_N'\right|^{3} = O\left(N^{-3 / 4}\right).
    \end{equation*}
    Therefore, in the statement of Theorem \ref{stein_special}, $\delta = O(N^{-1/4})$. Now, define, $\psi_{T_N}(x)=-\frac{\psi(x)}{\E[T_N\psi(T_N)]}$. As shown in  Proposition 1.4 of \cite{stein2004use}, the identity \eqref{2p14} is characterized by the density $p_{T_N}$ which is the solution to the differential equation $\frac{p'_{T_N}}{p_{T_N}}=\psi_{T_N}$. Hence, $p_{T_N}\propto \exp(-\frac{t^4}{4 \E T_N^4 })$.
    If $Z_{T_N}$ be distributed as $p_T$, then by Theorem \ref{stein_special}, we have
    \begin{equation*}
        \left|\E g(T_N)-\E g\left(Z_{T_N}\right)\right| \leq \frac{C}{N^{1 / 4}},
    \end{equation*}
    for every uniformly $1$-Lipschitz function $g$, where $C=C(\beta,h,p,q)$ is a constant. Next, note that there exists a constant $\kappa_q>0$ such that $\left|T_N'-T_N\right| \le \kappa_q N^{-3 / 4}$. Define $A := 2\kappa_q N^{-3/4}$. Then, $A\e|\psi(\bm T_N)| = O(N^{-3/4})$. Also, $A^3/\lambda = O(N^{-3/4})$, and finally, $\sqrt{E R(T_N)^2}/\lambda = N^{-1/4}$.
    This completes the proof of the first part of Theorem \ref{special_one}.

    We now prove the second part of Theorem \ref{special_one}. Since $V_1 = 0$, we look at $\bm V$ from its second coordinate onwards. Let $\cF$ be as defined before.  
    It follows from Lemma \ref{fourthTaylor}, that for every $r\ge 2$,
    \begin{equation*}
        \begin{aligned}
            \E\left[V_{r}'-V_{r} \mid \mathcal{F}\right] & =\E\left[W_{r}'-W_{r} \mid \mathcal{F}\right]                                                                                                                    \\
                                                         & = - \frac{1}{\sqrt{N}} \frac{1}{\beta{p(p-1)} \Xrb^{p-2}} \nabla_r G_{\beta, h}\left(\bm m_*+\frac{T \bm u}{N^{1 / 4}}+\frac{\bm V}{\sqrt{N}}\right)+R_{N}^{(1)} \\
                                                         & =\frac{1}{N}\frac{x_{s,q}^{p-2}}{\Xrb^{p-2}}\left[-1+\beta p(p-1) x_{s, q}^{p-1}\right] V_{r}+ \tilde{R} + R_N^{(1)}.
        \end{aligned}
    \end{equation*}
    Hence,
    \begin{equation*}
        \begin{aligned}
            \tilde{R}                   & =O\left(\frac{TV_r}{N^{5/4}}\right)+O\left(\frac{T^{3}}{N^{5 / 4}}\right)+O\left(\frac{T^{4}}{N^{3 / 2}}\right)+O\left(\sum_r \frac{|V_{r}| T}{N^{5 / 4}}\right) +O\left(\frac{1}{\sqrt{N}} \frac{1}{N} \sum_r V_{r}^{2}\right) \\
            \implies \|\tilde{R}\|_\ell & =O(N^{-5/4})~\text{for every}~\ell.
        \end{aligned}
    \end{equation*}
    Also, we have $R_N^{(1)} = O(N^{-3/2})$.
    Now, just as in the proof of Theorem \ref{regular_three}, we have:
    \begin{equation*}
        \begin{aligned}
            \E\left[V_{r}'-V_{r} \mid \mathcal{F}\right] & =\frac{1}{N}\left[-1+\beta p(p-1) x_{s, q}^{p-1}\right] V_{r}+ \tilde{R} + R_N^{(1)}+ R_N^{(3)},
        \end{aligned}
    \end{equation*}
    where $\|R_N^{(3)}\|_\ell=O(N^{-5/4})$ for all $\ell$.
    Hence, $R:=\tilde{R} + R_N^{(1)} + R_N^{(3)}(i)$ satisfies $\|R\|_\ell = O(N^{-5/4})$ for all $\ell$.

    We will now apply Theorem \ref{stein_third}. Since $\lambda^{(i)}=O(N)$, we have $C=O(N^{-1/4})$. Also,
    \begin{equation*}
        \left|V_{r}'-V_{r}\right|=\cO(N^{-1/2})
    \end{equation*}
    which means that $B= O(N^{-1/2})$. Moreover,
    \begin{equation*}
        \begin{aligned}
            \mathrm{Var}\left( \E\left[\left(V_{i}'-V_{i}\right)\left(V_{j}'-V_{j}\right) \vert \cF\right]\right) & =\mathrm{Var}\left(\E\left[\left(W_{i}'-W_{i}\right)\left(W_{j}'-W_{j}\right) \vert \cF\right]\right) \\
                                                                                                                  & =O\left(N^{-5 / 2}\right).
        \end{aligned}
    \end{equation*}
    This shows that $A=O(N^{-1/4})$, thereby proving:
    $$|\e g(\bm V_N) - \e g(\Sigma^{1/2} \bm Z)| = O_{\beta,h,p,q,g}(N^{-1/4})~.$$ The supremum bound can be proved by an exactly similar technique as the proof of the second part of Theorem \ref{regular_three}. The proof is now complete upon centering $\bm V$ and by similar arguments sketched in the proof of Theorem \ref{regular_three}.
\end{proof}

\begin{proof}[Proof of Theorem \ref{special_two}]
    Let $\cF$ be as defined before. Let us first write down $F_N$ with respect to $\cs$ and $\bm m_*$:
    \begin{equation*}
        F=\frac{N^{1 / 6}}{(1-q)}\left(\bar{X}_{\cdot 1} -m_1\right)=\frac{1}{(1-q) N^{5/6}}\left(\sum_{i=1}^{N} X_{i, 1}-N m_{1}\right).
    \end{equation*}
    Once again, just like the construction of $\bm W'$ as in the proof for Theorem \ref{regular_three}, we construct $F'$, and note that:
    \begin{equation}\label{fourp14}
        \begin{aligned}
            \E\left[F'-F \mid \cF\right] & =\frac{1}{(1-q)N^{11 / 6}} \sum_{i=1}^{N} \E\left[X_{i, 1}'-X_{i, 1} \mid \cF\right]                                   \\
                                         & =\frac{1}{(1-q)N^{11 / 6}} \sum_{i=1}^{N} \E\left[X_{i, 1}' \mid \cF \right]-\frac{1}{N} F-\frac{m_1}{(1-q)N^{5 / 6}}.
        \end{aligned}
    \end{equation}
    Now by \eqref{g_deriv},
    \begin{multline}
        \label{ey_special_two}
        \frac{1}{(1-q)N^{11 / 6}} \sum_{i=1}^{N} \E\left[X_{i, 1}' \mid \cF\right]=\frac{1}{(1-q) N^{5 / 6}}\left(\frac{-1}{\beta p(p-1) \bar{X}_{\cdot 1}^{p-2}} \nabla_1 G_{\beta, h}\left(\cs\right)+ \bar{X}_{\cdot 1}\right)\\
        +\frac{1}{(1-q) N^{1 / 3}} R_{N}^{(1)}
    \end{multline}
    Hence, by Taylor expansion,
    \begin{equation}
        \label{g_apply_two}
        \begin{aligned}
            \nabla_1 \gf \left(\bm m_*+\frac{F \bm u}{N^{1 / 6}}\right) & =\frac{d(p, q) F^{5}}{N^{5 / 6}}+R_n^{(2)},
        \end{aligned}
    \end{equation}
    where $\|R_N^{(2)}\|_\ell=O(\frac{1}{N})$ for all $\ell$ and $d(p,q)$ is some non-zero constant multiple of $\f^{(6)}(s)$. Now, using \eqref{rn1_bound},
    \begin{equation*}
        \frac{1}{(1-q) N^{1 / 3}} R_{N}^{(1)}=O\left(\frac{1}{N^{11 / 6}} \right).
    \end{equation*}
    From \eqref{fourp14}, \eqref{ey_special_two} and \eqref{g_apply_two}, we get that,
    \begin{equation*}
        \begin{aligned}
            \E\left[F'-F \mid \cF\right] & =\frac{1}{(1-q) N^{5 / 6}}\cdot\frac{-1}{\beta p(p-1) \bar{X}_{\cdot 1}^{p-2}}\left(\frac{d(p, q) F^{5}}{N^{5 / 6}}+R_N^{(2)}\right)+\frac{1}{(1-q) N^{1 / 3}} R_{N}^{(1)} \\
                                         & =\frac{1}{(1-q) N^{5 / 6}}\cdot\frac{-1}{\beta p(p-1) m_1^{p-2}}\left(\frac{d(p, q) F^{5}}{N^{5 / 6}}+R_N^{(2)}\right)+\frac{1}{(1-q) N^{1 / 3}} R_{N}^{(1)}+R_n^{(3)},
        \end{aligned}
    \end{equation*}
    where $R_N^{(3)}:=\frac{1}{(1-q) N^{5 / 6}}\left(\frac{d(p, q) F^{5}}{N^{5 / 6}}+R_N^{(2)}\right)\left(\frac{1}{\beta p(p-1) m_1^{p-2}}-\frac{1}{\beta p(p-1) \bar{X}_{\cdot 1}^{p-2}}\right)$.
    Therefore, we get: $$\E\left[F'-F \mid \cF\right]=\lambda \psi(F)+R(F)~,$$ where $\lambda=O(N^{-5/3})$, $\psi(x)= x^5$, and
    \begin{equation*}
        R(F):= \frac{1}{(1-q) N^{5 / 6}}\frac{-1}{\beta p(p-1) m_1^{p-2}}R_N^{(2)}+\frac{1}{(1-q) N^{1 / 3}} R_{N}^{(1)}+R_N^{(3)}.
    \end{equation*}

    Clearly, $\|R_N^{(3)}\|_\ell=O(N^{-11/6})$ for all $\ell$. Hence, $\|R(F)\|_\ell=O(N^{-11/6})$ for all $\ell$. So,
    \begin{equation*}
        \frac{c_{1}+c_{2} \sqrt{\E\left[F^{2}\right]}}{\lambda} \sqrt{\E\left[R^{2}\right]}=O\left(N^{-1 / 6}\right).
    \end{equation*}
    Now, let us get the remaining bounds,
    \begin{equation}\label{4p18u}
        \E\left[\left(F'-F\right)^{2} \mid F\right]=\frac{1}{(1-q)^{2} N^{2 / 3}} \E\left[\left(W_{1}'-W_{1}\right)^{2} \mid W\right].
    \end{equation}
    Again, following the proof of Theorem \ref{regular_three}, we see that $$\mathrm{Var}~ \E\left[\left(W_{1}'-W_{1}\right)^{2} \mid W\right] = O(N^{-7/3})$$
    and hence, from \eqref{4p18u}, we get:

    $$\mathrm{Var}~\E\left[\left(F'-F\right)^{2} \mid F\right] = O(N^{-11/3})~.$$ Therefore,
    \begin{equation*}
        \frac{c_{2}}{\lambda} \left(\Var~\E\left[\left(F'-F\right)^{2} \mid F\right]\right)^{1 / 2}=O\left(N^{-1 / 6}\right).
    \end{equation*}
    Since $|X_I-X'_I|$ is $O(1)$, we have $\frac{c_{3}}{\lambda} \E\left|F'-F\right|^{3}=O\left(N^{-5 / 6}\right)$. Therefore, in the statement of Theorem \ref{stein_special}, $\delta = O(N^{-1/6})$.  Define, $\psi_{F}(x)=-\frac{\psi(x)}{\E[F\psi(F)]}$. Now, the identity \eqref{2p14} is characterized by $p_F$ such that $\frac{p'_F}{p_F}=\psi_F$, as showed in Proposition 1.4 of \cite{stein2004use}. Hence, $p_F(x) \propto \exp(-\frac{x^6}{6 \E F^6})$.
    If $Z_{F_N}$ be distributed as $p_F$, then by Theorem \ref{stein_special},
    \begin{equation*}
        \left|\E g(F_N)-\E g\left(Z_{F_N}\right)\right| \leq \frac{C}{N^{1 / 6}},
    \end{equation*}
    for every uniformly 1-Lipschitz function g, where C = $C(\beta,h,p,q)$ is a constant. Next, note that there exists a constant $\kappa_q>0$ such that $\left|F'-F\right| \le \kappa_q N^{-5/6}$. Define $A := 2\kappa_q N^{-5/6}$. Then, $A\e|\psi(\bm F)| = O(N^{-5/6})$. Also, $A^3/\lambda = O(N^{-5/6})$, and finally, $\sqrt{E R(F_N)^2}/\lambda = N^{-1/6}$.
    This completes the proof of Theorem \ref{special_two}.
\end{proof}

\begin{proof}[Proof of Theorem \ref{thm:bempl}]
    Denote $\ell(\beta|\bm X) := \frac{1}{Np}\log L(\beta|\bm X)$. Then, note that:
    \begin{equation*}
        S(\cs,\beta) := \frac{\partial}{\partial \beta} \ell (\beta | \bm X)= ||\cs||_p^p-\frac{\sum_{r=1}^{q}\Xrb^{p-1}\exp(\beta p \Xrb^{p-1})}{\sum_{r=1}^{q}\exp(\beta p \Xrb^{p-1})}.
    \end{equation*}
    Also, note that:
    \begin{equation*}
        \frac{\partial^2}{\partial \beta^2} \ell (\beta | \bm X)= -\frac{\sum_{r=1}^{q}p\Xrb^{2p-2}\exp(\beta p \Xrb^{p-1})}{\sum_{r=1}^{q}\exp(\beta p \Xrb^{p-1})}+p\left(\frac{\sum_{r=1}^{q}\Xrb^{p-1}\exp(\beta p \Xrb^{p-1})}{\sum_{r=1}^{q}\exp(\beta p \Xrb^{p-1})}\right)^2.
    \end{equation*}
    It is easy to check that $\frac{\partial^2}{\partial \beta^2} L(\beta | \bm X)<0$, and hence the solution of $S(\cs, \beta)=0$ is the MPL estimator of $\beta$. Also, observe that:
    $$S(\bm m_i, \beta)= \frac{1}{\beta p(p-1)}\sum_{r=1}^q \nabla_r G_{\beta,0}(\bm m_i) = 0$$ for all maximizers $\bm m_i$ of $\hf$. Let $A_i$ be a neighborhood of $\bm m_i$ whose closure is devoid of any other maximizer. Now, by the implicit function theorem, there exists a neighborhood $U_i$ of $\bm m_i$, and a unique smooth function $g_i: U_i \to \R$, satisfying $g_i(\bm m_i)=\beta$ and $S(\bm x, g_i(\bm x))=0$ for all $\bm x\in U_i$. Define $B_i=A_i \cap U_i$.
  Again from the implicit function theorem, we get:
    \begin{equation}
        \label{grad_g_temp}
        \nabla g_i(\bm m_i) = -\frac{1}{\frac{\partial}{\partial \beta} S(\bm m_i,\beta)} \nabla_{\bm x} S(\bm m_i, \beta)
    \end{equation}  
  By a third-order Taylor expansion, we have the following on the event $\{\bar{\bm X}_N \in B_i\}$:
    $$\sqrt{N}(\hat{\beta}-\beta) = \sqrt{N}(g_i(\cs)-g_i(\bm m_i)) = \nabla g_i(\bm m_i) \cdot \bm W^{(i)} + \frac{1}{\sqrt N}\bm W^{(i)\top }\nabla^2 g_i(\bm m_i) \bm W^{(i)} + Q$$ where $Q := O\left(\frac{||\bm W^{(i)}||^3}{N}\right)$. Denoting by $\p_i$ the conditional measure given $\bar{\bm X}_N \in B_i$, we have:
    $$\p_i\left(\sqrt{N}(\hat{\beta} - \beta)>t\right) \leq \p_i\left(\nabla g_i(\bm m_i) \cdot \bm W^{(i)}+ \frac{1}{\sqrt N}\bm W^{(i)\top }\nabla^2 g_i(\bm m_i) \bm W^{(i)} > t-\frac{1}{\sqrt{N}}\right) + \p_i\left(Q>\frac{1}{\sqrt{N}}\right)$$
    and
    \begin{equation}\label{bhateq}
        \p_i\left(\sqrt{N}(\hat{\beta} - \beta)>t\right) \geq \p_i\left(\nabla g_i(\bm m_i) \cdot \bm W^{(i)}+ \frac{1}{\sqrt N}\bm W^{(i)\top }\nabla^2 g_i(\bm m_i) \bm W^{(i)} > t+\frac{1}{\sqrt{N}}\right) - \p_i\left(Q\leq -\frac{1}{\sqrt{N}}\right).
    \end{equation}
    Since the moments of $\bm W^{(i)}$ are bounded under $\p_i$, one has:
    $$\p_i\left(|Q| \geq \frac{1}{\sqrt{N}}\right) \leq N \e_{i} (Q^2) = O\left(\frac{1}{N}\right).$$
    Now, from Corollary \ref{actualBEcr}, we have:
    \begin{eqnarray}
        \label{eq:WZ}
       && \left\lvert\p_i\left(\nabla g_i(\bm m_i) \cdot \bm W^{(i)}+ \frac{1}{\sqrt N}\bm W^{(i)\top }\nabla^2 g_i(\bm m_i) \bm W^{(i)} > t\right)-\p\left(\nabla g_i(\bm m_i) \cdot \bm Z_N+ \frac{1}{\sqrt N}\bm Z_N^T\nabla^2 g_i(\bm m_i) \bm Z_N > t\right)\right\rvert\nonumber\\ &=& O\left(\frac{\log N}{\sqrt{N}}\right),
    \end{eqnarray}
    where $\bm Z_N=\Sigma_i^{1/2}\bm Z$ and $\bm Z$ is standard normal.
    Choose $L>0$ such that,
    \begin{equation*}
        \max_{1\le j\le q}\|\Sigma_{i}^{1/2 (j)}\|_2\le L,
    \end{equation*}
    where $\Sigma_{i}^{1/2 (j)}$ denotes the $j^{\mathrm{th}}$ row of $\Sigma_i^{1/2}$.
    Now, define $\alpha_N := L \sqrt{q\log N}$, and note that:
    $$\p(|\bm Z_{N,j}|>\alpha_N) \leq 2q\exp\left(-\frac{\alpha_N^2}{2qL^2}\right) = \frac{2q}{\sqrt{N}}~.$$
    Hence,
    \begin{eqnarray}\label{st41}
        && \left|\p(\nabla g_i(\bm m_i) \cdot \bm Z_N+ \frac{1}{\sqrt N}\bm Z_N^T\nabla^2 g_i(\bm m_i) \bm Z_N > x) - \p(\nabla g_i(\bm m_i)\cdot \bm Z_N>x) \right|\nonumber                                                                                       \\
        & \leq &  \left|\p(\nabla g_i(\bm m_i) \cdot \bm Z_N+ \frac{1}{\sqrt N}\bm Z_N^T\nabla^2 g_i(\bm m_i) \bm Z_N > x, ||\bm Z_N||_\infty \leq \alpha_N) - \p(\nabla g_i(\bm m_i)\cdot \bm Z_N>x, ||\bm Z_N||_\infty \leq \alpha_N)\right|     \nonumber                 \\
        &&+ 2\p(||\bm Z_N||_\infty>\alpha_N)\nonumber                                                                                                                                                                                               \\
        & \leq&  \left|\p(\nabla g_i(\bm m_i) \cdot \bm Z_N+ \frac{1}{\sqrt N}\bm Z_N^T\nabla^2 g_i(\bm m_i) \bm Z_N > x, ||\bm Z_N||_\infty \leq \alpha_N) - \p(\nabla g_i(\bm m_i)\cdot \bm Z_N>x, ||\bm Z_N||_\infty \leq \alpha_N)\right|\nonumber\\
        &&+ \frac{4q^2}{\sqrt{N}}.
    \end{eqnarray}  
    Next, denoting $b$ to be the spectral norm of $\nabla^2g_i(\bm m_i)$, we have:
    \begin{eqnarray*}
        &&\p\left(\nabla g_i(\bm m_i)\cdot \bm Z_N>x + \frac{bq\alpha_N^2}{\sqrt{N}}, \|\bm Z_N\|_\infty\leq \alpha_N\right) \\ &\le& \p\left(\nabla g_i(\bm m_i) \cdot \bm Z_N+ \frac{1}{\sqrt N}\bm Z_N^T\nabla^2 g_i(\bm m_i) \bm Z_N > x, ||\bm Z_N||_\infty \leq \alpha_N\right) \\ &\le& \p\left(\nabla g_i(\bm m_i)\cdot \bm Z_N>x - \frac{bq\alpha_N^2}{\sqrt{N}}, ||\bm Z_N||_\infty \leq \alpha_N\right).
    \end{eqnarray*}
    Also, since the density functions of $\bm Z_N$ is bounded by a constant independent of $N$,
    \begin{equation*}
        \begin{aligned}
             & \max_{r\in \{-1,1\}} \left|\p(\nabla g_i(\bm m_i)\cdot \bm Z_N>x + (-1)^r t_N~,~ ||\bm Z_N||_\infty \leq \alpha_N) - \p(\nabla g_i(\bm m_i)\cdot \bm Z_N>x , ||\bm Z_N||_\infty \leq \alpha_N)\right| \\
             & =O(t_N)
        \end{aligned}
    \end{equation*}
    for every sequence $t_N$.
   Combining all these, we have:
    \begin{eqnarray}\label{st42}
       && \left|\p(\nabla g_i(\bm m_i) \cdot \bm Z_N+ \frac{1}{\sqrt N}\bm Z_N^T\nabla^2 g_i(\bm m_i) \bm Z_N > x, ||\bm Z_N||_\infty \leq \alpha_N) - \p(\nabla g_i(\bm m_i)\cdot \bm Z_N>x , ||\bm Z_N||_\infty \leq \alpha_N)\right|\nonumber\\ &=& O\left(\frac{\log N}{\sqrt{N}}\right) .
    \end{eqnarray}
    Combining \eqref{st41} and \eqref{st42}, we thus get:
    \begin{align}\label{eq:Z}
        \sup_{x\in \mathbb{R}}\left|\p(\nabla g_i(\bm m_i) \cdot \bm Z_N+ \frac{1}{\sqrt N}\bm Z_N^T\nabla^2 g(\bm m) \bm Z_N > x) - \p(\nabla g_i(\bm m_i)\cdot \bm Z_N>x) \right| = O\left(\frac{\log N}{\sqrt{N}}\right) .
    \end{align}
    Therefore, from \eqref{eq:WZ} and \eqref{eq:Z},
    \begin{align}\label{beressquad}
        \sup_{x\in \mathbb{R}}\left|\p_i\left(\nabla g_i(\bm m_i) \cdot \bm W^{(i)}+ \frac{1}{\sqrt N}\bm W^{(i)\top }\nabla^2 g_i(\bm m_i) \bm W^{(i)} > x\right) - \p(\nabla g_i(\bm m_i)\cdot \bm Z_N>x) \right| = O\left(\frac{\log N}{\sqrt{N}}\right) .
    \end{align}
    In view of \eqref{bhateq} and \eqref{beressquad}, we thus have:
    \begin{multline*}
        \p\left(\nabla g_i(\bm m_i)^T \Sigma_i^{1/2}\bm Z > t +\frac{1}{\sqrt{N}}\right) - O\left(\frac{\log N}{\sqrt{N}}\right) \leq  \p_i\left(\sqrt{N}(\hat{\beta}-\beta)>t\right)\\
        \leq \p\left(\nabla g_i(\bm m_i)^T \Sigma_i^{1/2}\bm Z > t -\frac{1}{\sqrt{N}}\right) + O\left(\frac{\log N}{\sqrt{N}}\right) ,
    \end{multline*}
    where the hidden constants in the $O(\cdot)$ terms depend only on $\beta, p$ and $q$.
    Once again, since the density of $\nabla g_i(\bm m_i)^T \Sigma_i^{1/2}\bm Z$ is bounded by a constant independent of $N$, one has:
    $$\sup_{t\in \mathbb{R}} ~ \left|\p\left(\nabla g_i(\bm m_i)^T \Sigma_i^{1/2}\bm Z>t\pm \frac{1}{\sqrt{N}}\right) - \p(\nabla g_i(\bm m_i)^T \Sigma_i^{1/2}\bm Z>t) \right| = O\left(\frac{1}{\sqrt{N}}\right).$$
    This shows that:
    \begin{align}\label{bepnt}
        \sup_{t\in \mathbb{R}}~ \left|\p_i\left(\sqrt{N}(\hat{\beta}-\beta)>t\right)~-~\p(\nabla g_i(\bm m_i)^T \Sigma_i^{1/2}\bm Z>t)\right|~=~O\left(\frac{\log N}{\sqrt{N}}\right).
    \end{align}
   Now, it follows from Lemma \ref{crit_prob_bd}, $\P(\cs \in B_i) - \frac{1}{q} = O(N^{-\frac{1}{2}})$. This, together with the fact that $\p(\bar{\bm X}_N\in (\cup B_i)^c) $ is exponentially small (see Theorem 2.1 in \cite{bhowal_mukh}), proves Theorem \ref{thm:bempl} in view of \eqref{grad_g_temp}.
\end{proof}

\appendix
\section{Technical Lemmas}
In this appendix, we prove some technical lemmas necessary for showing the main results of the paper. We start by establishing a certain type of duality between the functions $G_{\beta,h}$ and $H_{\beta,h}$.

\begin{lemma}
    \label{hf_gf}
    The maximizers of $\hf$ are minimizers of $\gf$ with positive coordinates and vice versa.
\end{lemma}
\begin{proof}
    Let $\bm m= (m_1,m_2, \cdots ,m_q)$ be a maximizer of $\hf$, and $\lambda$ be the Lagrange multiplier for the constraint $\sum_{r=1}^{q}x_r=1$. Then, for each $i$ we have the following,
    \begin{equation*}
        \begin{aligned}
                        & \beta pm_r^{p-1} + h\delta_{r,1}-1-\log m_r = \lambda                                  \\
            \Rightarrow & \exp \left(\beta p m_{r}^{p-1}+h \delta_{r, 1}-1-\lambda\right)=m_{r}                  \\
            \Rightarrow & \sum_{r=1}^{q} \exp \left(\beta p m_{r}^{p-1}+h \delta_{r, 1}\right)=\exp (\lambda+1).
        \end{aligned}
    \end{equation*}
    This shows that,
    \begin{equation*}
        m_{r}=\frac{\exp \left(\beta p m_{r}^{p-1}+h \delta_{r, 1}\right)}{\sum_{s=1}^{q} \exp \left(\beta p m_s^{p-1}+h \delta_{s, 1}\right)}.
    \end{equation*}
    Hence, we have:
    \begin{equation*}
        \begin{aligned}
            H_{\beta, h}(\bm m) & = \beta \sum_{r=1}^{q}   m_{r}^{p}+h m_{1}-\sum_{r=1}^{q} m_{r} \log \left(\frac{\exp \left(\beta p m_{r}^{p-1}+h \delta_{i, 1}\right)}{\sum_{s=1}^{q} \exp \left(\beta p m_s^{p-1}+h \delta_{s, 1}\right)}\right) \\
                                & = \beta \sum_{r=1}^{q} m_{r}^{p}+h m_{1}-\sum_{r=1}^{q} m_{r}\left(\beta p m_{r}^{p-1}+h \delta_{r, 1}\right)+\log \left(\sum_{s=1}^{q} \exp \left(\beta p m_s^{p-1}+h \delta_{s, 1}\right)\right)                 \\
                                & =-\beta (p-1) \sum_{r=1}^{q} m_{r}^{p}+\log \left(\sum_{s=1}^{q} \exp \left(\beta p m_s^{p-1}+h \delta_{s, 1}\right)\right)                                                                                        \\
                                & = -G_{\beta, h}(\bm m).
        \end{aligned}
    \end{equation*}
    Now, let $\bm u$ be a minimizer of $G_{\beta, h}$ such that $u_{r}>0$ for all $r$. Hence, once again by the Lagrangian method, there exists $\lambda$ such that for each $r$,
    \begin{equation}\label{prod3}
        \begin{aligned}
            \beta p (p-1)u_r^{p-2} \left(u_{r}-\frac{\exp \left(\beta p u_{r}^{p-1}+h \delta_{r, 1}\right)}{\sum_{s=1}^{q} \exp \left(\beta p u_{s}^{p-1}+h \delta_{s, 1}\right)}\right)= \lambda
        \end{aligned}
    \end{equation}
    which implies that:
    \begin{equation*}
        \frac{\lambda}{\beta p (p-1)u_r^{p-2}}= u_{r}-\frac{\exp \left(\beta p u_{r}^{p-1}+h \delta_{r, 1}\right)}{\sum_{s=1}^{q} \exp \left(\beta p u_{s}^{p-1}+h \delta_{s, 1}\right)}.
    \end{equation*}
    Summing the above expression over all $r$, we get $\lambda=0$.

    It now follows from \eqref{prod3}, that:
    $$
        \begin{gathered}
            \quad u_{r} \log u_{r}=\beta p u_{r}^{p}+h u_{r} \delta_{r, 1}-u_{r} \log \left(\sum_{s=1}^{q} \exp \left(\beta p u_{s}^{p-1}+h \delta_{s, 1}\right)\right) \\
            \implies \sum_{r=1}^q u_{r} \log u_{r}=\beta p\|\bm u\|_{p}^{p}+h u_{1}-\log \left(\sum_{s=1}^{q} \exp \left(\beta p u_{s}^{p-1}+h \delta_{s, 1}\right)\right) \\
            \Rightarrow H_{\beta, h}(\bm u)=-G_{\beta, h}(\bm u)
        \end{gathered}
    $$


    Now, let $\bm m$ be a maximizer of $\hf$ and $\bm u$ be a minimizer of $\gf$ with positive coordinates. Then, by definition, we have $\hf(\bm m) \ge \hf(\bm u)$ and $\gf(\bm u)\le \gf(\bm m)$. The first inequality gives $-G_{\beta,h}(\bm m) \ge -G_{\beta,h}(\bm u)$, i.e. $\bm m$ is also a minimizer of $\gf$. Of course, by Proposition F.1 (iii) of \cite{bhowal_mukh}, all coordinates of $\bm m$ are positive. Similarly, the second inequality gives $-\hf(\bm u) \le -\hf(\bm m)$, i.e. $\bm u$ is also a maximizer of $\hf$. This proves Lemma \ref{hf_gf}.
\end{proof}

Next, we show a few properties of the matrix $\Lambda := [\Lambda_1,\ldots, \Lambda_q]^\top$, where $\Lambda_r := \nabla G_r(\bm m)$, with $G_r(\bm x) := x_r^{2-p}\nabla_r G_{\beta,h}(\bm x)$ and $\bm m$ being a minimizer of $\gf$.
\begin{lemma}[Properties of $\Lambda$]
    \label{proplamb}
    \begin{enumerate}
        \item We have:
              \begin{equation*}
                  \Lambda=\begin{bmatrix}
                      a      & b      & b      & b & \ldots & b \\
                      b'     & d      & c      & c & \ldots & c \\
                      b'     & c      & d      & c & \ldots & c \\
                      b'     & c      & c      & d &        & c \\
                      \vdots & \vdots & \vdots &   & \ddots &   \\
                      b'     & c      & c      & c & \ldots & d
                  \end{bmatrix}=
                  \begin{bmatrix}
                      a              & b \bm 1_{q-1}^T       \\
                      b' \bm 1_{q-1} & (d-c)I_{q-1}+cJ_{q-1}
                  \end{bmatrix}
              \end{equation*}
              where, $\bm 1_{q-1} \in \R^{q-1}$ is the all one vector, $J_{q-1}$ is a $(q-1) \times (q-1)$ matrix with all ones, and
              \begin{equation*}
                  \begin{aligned}
                      a  & :=\frac{\partial}{\partial u_{1}} \frac{1}{u_{1}^{p-2}} \frac{\partial}{\partial u_{1}} G_{\beta, h}(u)\Big|_{\bm u = \bm m}=\beta p(p-1)-\beta^{2} p^{2}(p-1)^{2}(q-1) m_1^{p-1} m_q                   \\
                      b  & :=\frac{\partial}{\partial u_{q}} \frac{1}{u_{1}^{p-2}} \frac{\partial}{\partial u_{1}} G_{\beta, h}(u)\Big|_{\bm u = \bm m}=\beta^{2}(p-1)^{2} p^{2} m_1 m_q^{p-1}                                     \\
                      b' & :=\frac{\partial}{\partial u_{1}} \frac{1}{u_{q}^{p-2}} \frac{\partial}{\partial u_{q}} G_{\beta, h}(u)\Big|_{\bm u = \bm m} =\beta^{2}(p-1)^{2} p^{2} m_1^{p-1} m_q                                    \\
                      c  & :=\frac{\partial}{\partial u_{2}} \frac{1}{u_{q}^{p-2}} \frac{\partial}{\partial u_{q}} G_{\beta, h}(u)\Big|_{\bm u = \bm m} =\beta^{2}(p-1)^{2} p^{2} m_q^{p}                                          \\
                      d  & :=\frac{\partial}{\partial u_{q}} \frac{1}{u_{q}^{p-2}} \frac{\partial}{\partial u_{q}} G_{\beta, h}(u)\Big|_{\bm u = \bm m} =\beta p(p-1)-\beta^{2} p^{2}(p-1)^{2} m_q^{p-1}\left(m_1+(q-2) m_q\right) \\
                         & =\beta p(p-1)\left[1-\beta p(p-1) m_q^{p-1}\left(1-m_q\right)\right]
                  \end{aligned}
              \end{equation*}

        \item $\det(\Lambda)=(d-c)^{q-2}\left[a((q-2) c+d)-(q-1) b b'\right]$
        \item Rank($\Lambda$)$\geq q-1$.
        \item Rank($\Lambda$)$=q-1$ iff $(\beta,h)$ is a special point. Moreover, the null space is given by $\operatorname{Span}(\bm u)$ where $\bm u = (1-q,1,1,\cdots,1)$
    \end{enumerate}
\end{lemma}
\begin{proof}
    \begin{enumerate}
        \item To begin with, note that:
              \begin{equation}
                  \label{g_deriv}
                  \frac{-1}{\beta p(p-1) u_{r}^{p-2}} \frac{\partial}{\partial u_{r}} G_{\beta, h}(u)+u_{r}=\frac{\exp \left(p \beta u_{r}^{p-1}+h \delta_{r, 1}\right)}{\sum_{s=1}^{q} \exp \left(p \beta u_{s}^{p-1}+h \delta_{s, 1}\right)}.
              \end{equation}
              Next, we derive the Hessian of $\gf$ at a minimizer $\bm m = \bm x_s$. Towards this, note that:
              \begin{equation}
                  \label{partials}
                  \begin{aligned}
                      \frac{\partial^{2}}{\partial^{2} u_{1}} G_{\beta, h} \Big|_{\bm u=\bm m}           & = \beta p(p-1)^{2} x_{s, 1}^{p-2}-\beta^{2}(p-1)^{2} p^{2} x_{s, 1}^{2 p-3} x_{s, q}(q-1)-\beta(p-2)(p-1) p x_{s, 1}^{p-2}                      \\
                                                                                                         & = \beta p(p-1) x_{s, 1}^{p-2}-\beta^{2}(p-1)^{2} p^{2}(q-1) x_{s, 1}^{2 p-3} x_{s, q}                                                           \\
                                                                                                         & = x_{s, 1}^{p-2}\left[\beta p(p-1)-\beta^{2}(p-1)^{2} p^{2}(q-1) x_{s, 1}^{p-1} x_{s, q}\right]                                                 \\
                      \frac{\partial^{2}}{\partial u_{1} \partial u_{q}} G_{\beta, h}\Big|_{\bm u=\bm m} & =\beta^{2}(p-1)^{2} p^{2} x_{s, 1}^{p-1} x_{s, q}^{p-1}                                                                                         \\
                      \frac{\partial^{2}}{\partial u_{2} \partial u_{q}} G_{\beta, h}\Big|_{\bm u=\bm m} & =\beta^{2}(p-1)^{2} p^{2} x_{s, q}^{2 p-2}                                                                                                      \\
                      \frac{\partial^{2}}{\partial^{2} u_{q}} G_{\beta, h}\Big|_{\bm u=\bm m}            & =\beta p(p-1)^{2} x_{s, q}^{p-2}-\beta^{2}(p-1)^{2} p^{2} x_{s, q}^{2 p-3}\left(x_{s, 1}+(q-2) x_{s, q}\right)-\beta(p-2)(p-1) p x_{s, q}^{p-2} \\
                                                                                                         & =x_{s, q}^{p-2}\left[\beta p(p-1)-\beta^{2}(p-1)^{2} p^{2} x_{s, q}^{p-1}\left(x_{s, 1}+(q-2) x_{s, q}\right)\right].
                  \end{aligned}
              \end{equation}
              Finally, we derive the gradient of $\frac{1}{u_r^{p-2}}\nabla_r \gf(\bm u)$ at the point $\bm m = \bm x_s$:
              \begin{equation*}
                  \begin{aligned}
                      a  & :=\frac{\partial}{\partial u_{1}} \frac{1}{u_{1}^{p-2}} \frac{\partial}{\partial u_{1}} G_{\beta, h}(u)\Big|_{\bm u=\bm m} =\beta p(p-1)-\beta^{2} p^{2}(p-1)^{2}(q-1) x_{s,1}^{p-1} x_{s, q}                       \\
                      b  & :=\frac{\partial}{\partial u_{q}} \frac{1}{u_{1}^{p-2}} \frac{\partial}{\partial u_{1}} G_{\beta, h}(u)\Big|_{\bm u=\bm m}  =\beta^{2}(p-1)^{2} p^{2} x_{s,1} x_{s, q}^{p-1}                                        \\
                      b' & :=\frac{\partial}{\partial u_{1}} \frac{1}{u_{q}^{p-2}} \frac{\partial}{\partial u_{q}} G_{\beta, h}(u)\Big|_{\bm u=\bm m} =\beta^{2}(p-1)^{2} p^{2} x_{s,1}^{p-1} x_{s, q}                                         \\
                      c  & :=\frac{\partial}{\partial u_{2}} \frac{1}{u_{q}^{p-2}} \frac{\partial}{\partial u_{q}} G_{\beta, h}(u)\Big|_{\bm u=\bm m} =\beta^{2}(p-1)^{2} p^{2} x_{s, q}^{p}                                                   \\
                      d  & :=\frac{\partial}{\partial u_{q}} \frac{1}{u_{q}^{p-2}} \frac{\partial}{\partial u_{q}} G_{\beta, h}(u)\Big|_{\bm u=\bm m} =\beta p(p-1)-\beta^{2} p^{2}(p-1)^{2} x_{s, q}^{p-1}\left(x_{s,1}+(q-2) x_{s, q}\right) \\
                         & =\beta p(p-1)\left[1-\beta p(p-1) x_{s, q}^{p-1}\left(1-x_{s, q}\right)\right].
                  \end{aligned}
              \end{equation*}

        \item It is easy to check that $\Lambda$ is similar to the following matrix,
              \begin{equation*}
                  \left[\begin{array}{cccccc}
                          a        & b         & 0      & 0      &        & 0      \\
                          (q-1) b' & (q-2) c+d & 0      & 0      & \cdots & 0      \\
                          b'       & c         & d-c    & 0      &        & 0      \\
                          \vdots   & \vdots    & \vdots & \vdots &        & \vdots \\
                          b'       & c         & 0      & 0      & \cdots & d-c
                      \end{array}\right].
              \end{equation*}
              Hence, $\det(\Lambda)=(d-c)^{q-2}\left[a((q-2) c+d)-(q-1) b b'\right]$.
        \item First, suppose that $q>2$. Note that $d-c=\beta p(p-1)\left[1-\beta p(p-1) x_{s, q}^{p-1}\right]$. It follows from Proposition F.1 (iv) in \cite{bhowal_mukh} that $1-\beta p(p-1) x_{s, q}^{p-1} >0$. Again from Proposition F.1 (iii) in \cite{bhowal_mukh}, we have $b> 0$. This shows that all but (possibly) the second rows of $\Lambda$ are linearly independent, giving rank$(\Lambda) \ge q-1$. For $q=2$ this is trivial, since $\Lambda\ne 0$,

        \item In view of part (iii), we have rank($\Lambda$)$= q-1$ iff $\det(\Lambda)=0$ iff  $a((q-2) c+d)-(q-1) b b'=0$. Now, it follows by a straightforward algebra, that:
              \begin{equation*}
                  \begin{aligned}
                      a((q-2) c+d)-(q-1) b b' & = \beta p(p-1) \left(1-\beta p(p-1)(q-1) x_{s,1}^{p-1} x_{s, q}-\beta p(p-1) x_{s, q}^{p-1} x_{s,1}\right) \\
                                              & =-\beta p(p-1) f_{\beta, h}^{(2)}(s) \frac{x_{s,1} x_{s, q}}{q-1}.
                  \end{aligned}
              \end{equation*}
    \end{enumerate}
    Hence, $f_{\beta, h}^{(2)}(s) = 0$ as $x_{0,1} x_{0, q}>0$. This shows that $(\beta, h)$ is a special point.
    Furthermore, observe that $(1-q)a+(q-1)b=-\beta p(p-1)\f^{(2)}(s)$. Thus, at a special point $(1-q)a+(q-1)b=0$. Now,
    \begin{equation*}
        \operatorname{det} \begin{bmatrix}
            a       & b        \\
            (q-1)b' & (q-2)c+d
        \end{bmatrix} = 0.
    \end{equation*}
    Hence,$[a,b]=\kappa[(q-1)b', (q-2)c+d]$. Since, $a-b=0$ so $(1-q)b'+(q-2)c+d=0$. Note that,
    \begin{equation*}
        \Lambda \bm u = [(1-q)a+(q-1)b,(1-q)b'+(q-2)c+d,\ldots,(1-q)b'+(q-2)c+d]^T.
    \end{equation*}
    Hence, $\bm u$ lies in the nullspace. Since, the dimension of the nullspace is 1, it is exactly equal to $\operatorname{Span}(\bm u)$.
\end{proof}

The following lemma is easy to check by direct computation, and we ignore its proof.
\begin{lemma}
    For each $r\in \{1,\cdots,q\}$ we have
    \begin{equation*}
        \mathbb{P}\left(X_{j}=r \mid\left(X_{t}\right)_{t \neq j}\right)=\frac{\exp \left(p \beta m_{ j,r}^{p-1}+h \delta_{r, 1}\right)}{\sum_{s=1}^{q} \exp \left(p \beta m_{j,s}^{p-1}+h \delta_{s, 1}\right)},
    \end{equation*}
    where  $m_{i,r} := \frac{1}{N}\sum_{t\ne i} X_{t,r}$.
\end{lemma}


In the next lemma, we analyze the fourth order Taylor expansion of $G_{\beta,h}$, which is necessary for proving Theorem \ref{special_one}.

\begin{lemma}[Fourth order Taylor expansion]
    \label{fourthTaylor}
    Suppose that $(\beta,h)$ is a type-I special point. Let $\bm u =(1-q,1,\cdots,1)$, $\bm v\in \mathcal{H}_q \cap \bm u^\perp$ and $\bm m=\bm x_s$ be the unique maximizer of $H_{\beta,h}$. Then,
    \begin{equation*}
        \nabla_r G_{\beta,h}\left(\bm m+t \bm u+ \bm v\right)=A_1+ O\left(\sum_{k=1}^q v_{k}^{2}\right)+O\left(\max\{t,||v||_\infty\} \sum_{k=1}^{q} |v_{k}|\right)+t^{3} c(p, q)+O\left(t^{4}\right),
    \end{equation*}
    where, $A_1=0$ for $r=1$, and $A_1=\beta p(p-1)\left[1-\beta p(p-1) m_q^{p-1}\right] m_q^{p-2} v_{r}+O\left(t v_r\right)$ for $r \in \{2,\cdots, q\}$, $c(p,q)=C_{p,q}\f^{(4)}(s)\neq 0$, and $C_{p,q}$ is a non-zero constant depending on $p,q$.
\end{lemma}
\begin{proof}
    For notational convenience, let us define:
    \begin{equation*}
        \partial_{r_{1}, \ldots, r_{p}}(\bm z):=\frac{\partial^{p} G_{\beta,h}}{\partial x_{r_{1}} \cdots \partial x_{r_{p}}}(\bm z).
    \end{equation*}
    Note that $v_1=0$ and $\sum_{i=2}^{q}v_i=0$. First, note that by Taylor expansion, we have:
    \begin{multline}
        \label{taylor_xtu}
        \frac{\partial G_{\beta,h}}{\partial x_{r}}\left(\bm m+t \bm u+ \bm v\right)=\frac{\partial \gf}{\partial x_{r}}(\bm m + t \bm u)+\sum_{s=1}^{q} \frac{\partial}{\partial x_{s}} \frac{\partial}{\partial x_{r}} \gf (\bm m+t \bm u) v_{s}                                                                          \\
        +\frac{1}{2} \sum_{s, \ell=1}^{q} \frac{\partial}{\partial x_{\ell}} \frac{\partial}{\partial x_{s}} \frac{\partial}{\partial x_{r}} \gf\left(\bm m +t \bm u+\gamma_{r}\bm v\right) v_{s} v_{t}.
    \end{multline}
    By a further Taylor expansion, we have:
    \begin{equation}
        \label{taylor_x}
        \frac{\partial G_{\beta,h}}{\partial x_{r}} (\bm m+t \bm u)=\frac{t^{3}}{6} \sum_{s,\ell,k} \partial_{r, s, \ell, k}\left(\bm m +\tilde{\gamma}_{r} t \bm u\right) u_{k} u_{l} u_{m}.
    \end{equation}
    So by \eqref{taylor_xtu} and \eqref{taylor_x}, we have:
    \begin{multline*}
        \frac{\partial G_{\beta,h} }{\partial x_{r}}  \left(\bm m+t \bm u+ \bm v\right)=\sum_{s=1}^{q} \frac{\partial}{\partial x_{s}} \frac{\partial}{\partial x_{r}} \gf (\bm m+t \bm u) v_{s}
        +\frac{1}{2} \sum_{s, \ell=1}^{q} \frac{\partial}{\partial x_{\ell}} \frac{\partial}{\partial x_{s}} \frac{\partial}{\partial x_{r}} \gf\left(\bm m +t \bm u+\gamma_{r}\bm v\right) v_{s} v_{\ell}. \\
        +\frac{t^{3}}{6} \sum_{s,\ell,k} \partial_{r, s, \ell, k}\left(\bm m +\tilde{\gamma}_{r} t \bm u\right) u_{k} u_{l} u_{m}~.
    \end{multline*}
    Let us rename the summands $A_1(r),A_2(r),A_3(r)$, respectively. Clearly, $A_1(1)=0$. For $r>1$, by a further Taylor expansion, we have:
    \begin{equation*}
        \begin{aligned}
            A_1(r) & =\sum_{s=2}^{q}\partial_{r,s}v_s + O\left(t\sum_{s=1}^{q}|v_s|\right)
                   & =\beta p(p-1)\left[1-\beta p(p-1) m_q^{p-1}\right] m_q^{p-2} v_{r}+O\left(t\sum_{s=1}^{q}|v_s|\right).
        \end{aligned}
    \end{equation*}
    The last equation follows from \eqref{partials}. Similarly,
    \begin{equation*}
        \begin{aligned}
            A_{2}(r) & =\frac{1}{2} \sum_{s, k=1}^{q} \frac{\partial}{\partial x_{k}} \frac{\partial}{\partial x_{s}} \frac{\partial}{\partial x_{r}} \gf\left(\bm m\right) v_{s} v_{k}+O\left(\max\{t,||v||_\infty\} \sum_{k=1}^{q} v_{k}^{2}\right) \\
                     & =O\left(\sum_{k=1}^q v_{k}^{2}\right)+O\left(\max\{t,||v||_\infty\} \sum_{k=1}^{q} v_{k}^{2}\right).
        \end{aligned}
    \end{equation*}
    Also, note that $A_{3}(r)=t^{3} c(p, q)+O\left(t^{4}\right)$ where $c(p,q)$ is a non-zero constant multiple of $\f^{(4)}(s)$ and hence non-zero at a type-I special point.
    The proof of Lemma \ref{fourthTaylor} is now complete.
\end{proof}
Recall that $\cP_q$ was the probability simplex in $\R^q$. Let $\cP_{q,N}$ denote the set of all vectors in $\cP_q$, all of whose entries have the form $i/N$ for some $0\le i\le N$.
\begin{lemma}\label{densappr}
    For $\bm v \in \cP_{q,N}$, we have:
    $$q^N Z_N(\beta,h) \p_{\beta,h,N}(\cs = \bm v) = \left(1+O\left(\frac{1}{N}\right)\right) N^{-\frac{q-1}{2}} A(\bm v) e^{N H_{\beta,h}(\bm v)}$$ where $A(\bm v) := (2\pi)^{-(q-1)/2} \prod_{r=1}^q v_r^{-1/2}$.
\end{lemma}

\begin{proof}
    For a $\bm v \in \cP_{q,N}$, we have:
    \begin{equation}\label{prst}
        q^N Z_N(\beta,h) \p_{\beta,h,N}(\cs = \bm v) = |A_N(\bm v)| \exp \left\{N \left(\beta  \sum_{r=1}^q v_r^p + hv_1\right)\right\}~.
    \end{equation}
    By Stirling's formula, we have:

    \begin{equation*}
        |A_N(\bm v)| = \frac{N!}{\prod_{r=1}^q (Nv_r)!} = (2\pi N)^{\frac{1-q}{2}}\left(\prod_{r=1}^q v_r^{-\frac{1}{2}}\right) e^{-N\sum_{r=1}^q v_r \log v_r}\left(1+O\left(\frac{1}{N}\right)\right)~,
    \end{equation*}
    where the $O\left(\frac{1}{N}\right)$ term is uniformly over all $\bm v \in \cP_{q,N}$.
    Therefore, we have from \eqref{prst},
    \begin{equation*}
        q^N Z_N(\beta,h) \p_{\beta,h,N}(\cs = \bm v) = (2\pi N)^{\frac{1-q}{2}}\left(\prod_{r=1}^q v_r^{-\frac{1}{2}}\right)e^{N H_{\beta,h}(\bm v)}\left(1+O\left(\frac{1}{N}\right)\right)~~.
    \end{equation*}
    This completes the proof of Lemma \ref{densappr}.
\end{proof}
\begin{lemma}
    \label{reg_expec_bd}
    Assume that $(\beta, h) \in \cR_{p,q}$ and let $\bm m_*$ be the unique maximizer of $H_{\beta, h}$. Then,
    $
        \|\e \bm W_N\|_\infty = O(N^{-1/2}).
    $
\end{lemma}
\begin{proof}
    Fix $\eps>0$ and consider a continuously differentiable function $g: \R^q \to \R$ which we will later choose suitably. For every $\bm v \in \cP_{q,N}$, define $\bm w(\bm v) = \bm w_N(\bm v) := \sqrt{N}(\bm v - \bm m_*)$. Throughout this proof, we take $\|\cdot \|$ to be the $L^\infty$ vector norm. Then, we have by Lemma \ref{densappr},
    \begin{eqnarray*}
        &&q^N Z_N(\beta,h) \e_{\beta,h,N}\left[g(\bm W_N) \mathbbm{1}_{\|\bm W_N\| \le \eps N^{1/7}}\right]\nonumber\\ &=&  \sum_{\bm v \in \cP_{q,N}} g(\bm w(\bm v)) \mathbbm{1}_{\|\bm w(\bm v)\| \le \eps N^{1/7}}~ q^N Z_N(\beta,h) \p_{\beta,h,N}(\cs = \bm v)\nonumber\\&=& \left(1+O\left(\frac{1}{N}\right)\right) N^{-\frac{q-1}{2}}  \sum_{\bm v \in \cP_{q,N}} A(\bm v) e^{N H_{\beta,h}(\bm v)} g(\bm w(\bm v)) \mathbbm{1}_{\|\bm w(\bm v)\| \le \eps N^{1/7}}\nonumber\\&=& \left(1+O\left(\frac{1}{N}\right)\right) N^{-\frac{q-1}{2}}  \sum_{\bm v \in \cP_{q,N}} A\left(\bm m_* + N^{-\frac{1}{2}} \bm w(\bm v)\right) e^{N H_{\beta,h}\left(\bm m_* + N^{-\frac{1}{2}} \bm w(\bm v)\right)}\nonumber\\ && g(\bm w(\bm v)) \mathbbm{1}_{\|\bm w(\bm v)\| \le \eps N^{1/7}}\nonumber\\&=& \left(1+O\left(\frac{1}{N}\right)\right) N^{-\frac{q-1}{2}}  A(\bm m_*) e^{N H_{\beta,h}(\bm m_*)} \sum_{\bm v \in \cP_{q,N}} \left(1+O\left(\frac{\|\bm w(\bm v)\|}{\sqrt{N}}\right)\right)\left(1+O\left(\frac{\|\bm w(\bm v)\|^3}{\sqrt{N}}\right)\right)\\&& e^{\frac{1}{2}\bm Q_{\bm m_*,\beta}(\bm w(\bm v))} g(\bm w(\bm v)) \mathbbm{1}_{\|\bm w(\bm v)\| \le \eps N^{1/7}}~,
    \end{eqnarray*}
    where the last step uses the following Taylor expansion:
    \begin{equation*}
        NH_{\beta, h}\left(\bm m_*+N^{-\frac{1}{2}} \bm w(\bm v)\right)=NH_{\beta,h}(\bm m_*) + \frac{1}{2} \bm Q_{\bm m_*,\beta}(\bm w (\bm v))+O\left(\frac{\|\bm w(\bm v)\|^3}{\sqrt N}\right).
    \end{equation*}
    Hence, we have:
    \begin{eqnarray*}
        &&q^N Z_N(\beta,h) \e_{\beta,h,N}\left[g(\bm W_N) \mathbbm{1}_{\|\bm W_N\| \le \eps N^{1/7}}\right]\nonumber\\ &=& \left(1+O\left(\frac{1}{N}\right)\right) N^{-\frac{q-1}{2}}  A(\bm m_*) e^{N H_{\beta,h}\left(\bm m_*\right)}\sum_{\bm v \in \cP_{q,N}} g(\bm w(\bm v)) \mathbbm{1}_{\|\bm w(\bm v)\| \le \eps N^{1/7}} e^{\frac{1}{2} \bm Q_{\bm m_*,\beta}(\bm w (\bm v))} + R_N,
    \end{eqnarray*}

    where

    $$R_N = O\left(\frac{e^{NH_{\beta,h}(\bm m_*)}}{\sqrt{N}}\right) N^{-\frac{q-1}{2}} \sum_{\bm v \in \cP_{q,N}} \|\bm w(\bm v)\|' g(\bm w(\bm v)) \mathbbm{1}_{\|\bm w(\bm v)\| \le \eps N^{1/7}} e^{\frac{1}{2} \bm Q_{\bm m_*,\beta}(\bm w (\bm v))},$$

    with $\|\bm w(\bm v)\|' := \max\{\|\bm w(\bm v)\|, \|\bm w(\bm v)\|^3, \|\bm w(\bm v)\|^4\}$.

    Now for any function $h:\mathbb{R}^q \to \mathbb{R}$ which grows at most polynomially with its coordinates,
    \begin{equation*}
        \begin{aligned}
                     & \left|\int_{\cH_q \bigcap B(\bm 0,\eps N^{1/7})} h(\bm w) e^{\frac{1}{2}  \bm Q_{\bm m_*,\beta} (\bm w)}~dw_1dw_2\ldots dw_q-N^{-\frac{q-1}{2}} \sum_{\bm v \in \cP_{q,N}} h(\bm w(\bm v)) \mathbbm{1}_{\|\bm w(\bm v)\| \le \eps N^{1/7}} e^{\frac{1}{2} \bm Q_{\bm m_*,\beta}(\bm w (\bm v))}\right| \\
            \le      & \sum_{\bm v \in \cP_{q,N}}
            \int_{\cH_q \cap \{\bm w:\|\bm w - \bm w(v)\|\leq 1/ 2\sqrt N\}} \left|h(\bm w) e^{\frac{1}{2}  \bm Q_{\bm m_*,\beta} (\bm w)}- h(\bm w(\bm v)) e^{\frac{1}{2} \bm Q_{\bm m_*,\beta}(\bm w (\bm v))}\right|~dw_1dw_2\ldots dw_q                                                                                   \\
            \lesssim & \sum_{\bm v \in \cP_{q,N}}
            \int_{\cH_q \cap \{\bm w:\|\bm w - \bm w(v)\|\leq 1/ 2\sqrt N\}} \frac{1}{\sqrt N}~dw_1dw_2\ldots dw_q                                                                                                                                                                                                            \\
            \lesssim & \sum_{\bm v \in \cP_{q,N}}
            \frac{1}{N^{\frac{q-1}{2}}} \frac{1}{\sqrt N}                                                                                                                                                                                                                                                                     \\
            =        & O\left(\frac{1}{\sqrt N}\right)~.
        \end{aligned}
    \end{equation*}
    Hence, by the Riemann sum approximation, we get $R_N= O(N^{-\frac{1}{2}}e^{NH_{\beta,h}(\bm m_*)})$ and therefore,
    \begin{equation}
        \label{cont_approx}
        \begin{aligned}
             & q^N Z_N(\beta,h) \e_{\beta,h,N}\left[g(\bm W_N) \mathbbm{1}_{\|\bm W_N\| \le \eps N^{1/7}}\right]                                                                                                                                                                 \\
             & =\left(1+O\left(\frac{1}{N}\right)\right) A(\bm m_*) e^{N H_{\beta,h}\left(\bm m_*\right)} \int_{\cH_q \bigcap B(\bm 0,\eps N^{1/7})} g(\bm w) e^{\frac{1}{2}  \bm Q_{\bm m_*,\beta} (\bm w)}~dw_1dw_2\ldots dw_q + O(N^{-\frac{1}{2}}e^{NH_{\beta,h}(\bm m_*)}).
        \end{aligned}
    \end{equation}

    Now, take $g\equiv 1$, whence we have:
    \begin{equation*}
        \begin{aligned}
             & q^N Z_N(\beta,h) \p_{\beta,h,N}\left(\|\bm W_N\| \le \eps N^{1/7}\right)                                                                                                                                                                                 \\
             & =\left(1+O\left(\frac{1}{N}\right)\right) A(\bm m_*) e^{N H_{\beta,h}\left(\bm m_*\right)} \int_{\cH_q \bigcap B(\bm 0,\eps N^{1/7})} e^{\frac{1}{2}  \bm Q_{\bm m_*,\beta} (\bm w)}~dw_1dw_2\ldots dw_q + O(N^{-\frac{1}{2}}e^{NH_{\beta,h}(\bm m_*)}).
        \end{aligned}
    \end{equation*}

    It follows from the proof of Theorem 2.1 in \cite{bhowal_mukh} that $$\p_{\beta,h,N}\left(\|\bm W_N\| > \eps N^{1/7}\right) = O\left(e^{-CN^{2/7}}\right)$$ for some constant $C>0$ and hence,
    \begin{equation}
        \label{partition_approx}
        \begin{aligned}
            q^N Z_N(\beta,h) & =\left(1+O\left(\frac{1}{N}\right)\right) A(\bm m_*) e^{N H_{\beta,h}\left(\bm m_*\right)} \int_{\cH_q \bigcap B(\bm 0,\eps N^{1/7})} e^{\frac{1}{2}  \bm Q_{\bm m_*,\beta} (\bm w)}~dw_1dw_2\ldots dw_q \\ &+ O(N^{-\frac{1}{2}}e^{NH_{\beta,h}(\bm m_*)}).
        \end{aligned}
    \end{equation}
    Hence, by \eqref{cont_approx} and \eqref{partition_approx}, we have:
    \begin{equation*}
        \e_{\beta,h,N}\left[g(\bm W_N) \mathbbm{1}_{\|\bm W_N\| \le \eps N^{1/7}}\right]= \frac{ \int_{\cH_q \bigcap B(\bm 0,\eps N^{1/7})} g(\bm w) e^{\frac{1}{2}  \bm Q_{\bm m_*,\beta} (\bm w)}~dw_1dw_2\ldots dw_q}{\int_{\cH_q \bigcap B(\bm 0,\eps N^{1/7})} e^{\frac{1}{2}  \bm Q_{\bm m_*,\beta} (\bm w)}~dw_1dw_2\ldots dw_q} ~+~ O(N^{-1/2}).
    \end{equation*}
    Therefore, for any odd function $g$ integrable with respect to the multivariate Gaussian measure,
    $$\e_{\beta,h,N}\left[g(\bm W_N) \mathbbm{1}_{\|\bm W_N\| \le \eps N^{1/7}}\right] = O(N^{-\frac{1}{2}})~.$$

    Now, taking $g$ to be the coordinate projections, we get
    \begin{equation*}
        \left|\e_{\beta,h,N} \left(\bm W_{N,r} \mathbbm{1}_{\|\bm W_N\| \le \eps N^{1/7}}\right)\right| = O(N^{-\frac{1}{2}})
    \end{equation*}
    On the other hand,
    $$\left|\e_{\beta,h,N} \left(\bm W_{N,r} \mathbbm{1}_{\|\bm W_N\| > \eps N^{1/7}}\right)\right| = O(\sqrt{N})\p_{\beta,h,N}\left(\|\bm W_N\| > \eps N^{1/7}\right) = O\left(N^{\frac{1}{2}} e^{-CN^{2/7}}\right) = O(N^{-\frac{1}{2}})~.$$
    This completes the proof of Lemma \ref{reg_expec_bd}.
\end{proof}
\begin{lemma}
    \label{reg_expec_bd_v2}
    Assume that $(\beta, h) \in \cS^1_{p,q}$ and let $\bm m_*$ be the unique maximizer of $H_{\beta, h}$. Then there exists a constant $C>0$ (depending on $\beta, h$ and $q$), such that,
    $$
        \left\|\E \bm V_N\right\|_\infty = O(N^{-1 / 4}).
    $$
\end{lemma}
The proof of Lemma \ref{reg_expec_bd_v2} follows similarly as the proof of Lemma \ref{reg_expec_bd}, so we skip it. Finally, we state and prove another lemma very similar to the last two, which will be crucial in proving Theorem \ref{thm:bempl}.

\begin{lemma}
    \label{crit_prob_bd}
    Assume that $(\beta, h) \in \cC_{p,q}$ and let $\{\bm m_i\}_{i=1}^K$ be the maximizers of $H_{\beta, h}$. Let $A_i$ be a neighborhood around $\bm m_i$ whose closure does not contain any other maximizer. Then,
    $$
        |\P(\cs \in A_i)-p_i| = O(N^{-1/2}),
    $$
    where
    \begin{equation*}
        p_k:=\frac{\tau(\bm m_k)}{\sum_{i=1}^{K}\tau(\bm m_i)},
    \end{equation*}
    \begin{equation}
        \tau(\bm m_i):=\sqrt{\frac{f^\dprime_{\beta,h}(s_i)^{-1}\left(-k^\dprime\left(\frac{1-s_i}{q}\right)\right)^{2-q}}{\prod_{r=1}^q m_{i,r}}}~,
        \label{tau_def}
    \end{equation}
   and $\bm m_i$ is a permutation of $\bm x_{s_i}$.
\end{lemma}
\begin{proof}
    From the proof of Lemma \ref{reg_expec_bd}, we get the following once we take $g(x)=\one_{x \in \sqrt{N} (A_i - \bm m_i)}$:
    \begin{equation}
        \label{cont_approx_crit}
        \begin{aligned}
             & q^N Z_N(\beta,h) \e_{\beta,h,N}\left[g(\bm W_N^{(i)}) \mathbbm{1}_{\|\bm W_N^{(i)}\| \le \eps N^{1/7}}\right]                                                                                                                                                     \\
             & =\left(1+O\left(\frac{1}{N}\right)\right) A(\bm m_i) e^{N H_{\beta,h}\left(\bm m_i\right)} \int_{\cH_q \bigcap B(\bm 0,\eps N^{1/7})} g(\bm w) e^{\frac{1}{2}  \bm Q_{\bm m_i,\beta} (\bm w)}~dw_1dw_2\ldots dw_q + O(N^{-\frac{1}{2}}e^{NH_{\beta,h}(\bm m_i)}),
        \end{aligned}
    \end{equation}
    where $\eps$ is sufficiently small.
    Hence we have,
    \begin{equation*}
        \begin{aligned}
             & q^N Z_N(\beta,h) \sum_{i=1}^{K}\p_{\beta,h,N} \left(\|\bm W_N^{(i)}\| \le \eps N^{1/7}\right)                                                                                                                                                                           \\
             & =\left(1+O\left(\frac{1}{N}\right)\right) \sum_{i=1}^{K} A(\bm m_i) e^{N H_{\beta,h}\left(\bm m_i\right)} \int_{\cH_q \bigcap B(\bm 0,\eps N^{1/7})} e^{\frac{1}{2}  \bm Q_{\bm m_i,\beta} (\bm w)}~dw_1dw_2\ldots dw_q + O(N^{-\frac{1}{2}}e^{NH_{\beta,h}(\bm m_i)}).
        \end{aligned}
    \end{equation*}
   Now, it follows from the proof of Theorem 2.1 in \cite{bhowal_mukh} that $$\p_{\beta,h,N}\left( \min_{i}\|\bm W_N^{(i)}\| > \eps N^{1/7}\right) = O\left(e^{-CN^{2/7}}\right)$$ for some constant $C>0$. Since the events $\{\|\bm W_N^{(i)}\| \le \eps N^{1/7}\}_{i=1}^K$ are asymptotically disjoint, we have:
    \begin{equation}
        \label{partition_approx_crit}
        \begin{aligned}
            q^N Z_N(\beta,h) & =\left(1+O\left(\frac{1}{N}\right)\right) \sum_{i=1}^{K} A(\bm m_i) e^{N H_{\beta,h}\left(\bm m_i\right)} \int_{\cH_q \bigcap B(\bm 0,\eps N^{1/7})} e^{\frac{1}{2}  \bm Q_{\bm m_i,\beta} (\bm w)}~dw_1dw_2\ldots dw_q \\ &+ O(N^{-\frac{1}{2}}e^{NH_{\beta,h}(\bm m_i)}).
        \end{aligned}
    \end{equation}
    Hence, by \eqref{cont_approx_crit} and \eqref{partition_approx_crit}, we have:
    \begin{equation*}
        \e_{\beta,h,N}\left[\one_{\cs \in A_i} \mathbbm{1}_{\|\bm W_N^{(i)}\| \le \eps N^{1/7}}\right]= \frac{\tau(\bm m_i)}{\sum_{i=1}^K \tau(\bm m_i)} ~+~ O(N^{-1/2}).
    \end{equation*}
    Therefore, we have:
    $$\p_{\beta,h,N}\left(\cs \in A_i~,~ \|\bm W_N^{(i)}\| \le \eps N^{1/7}\right) - p_i = O(N^{-\frac{1}{2}})~.$$
    On the other hand,
    $$\p_{\beta,h,N}\left(\cs \in A_i~,~ \|\bm W_N^{(i)}\| > \eps N^{1/7}\right) = O\left(\p_{\beta,h,N}\left(\|\bm W_N^{(i)}\| > \eps N^{1/7}\Big| \cs \in A_i\right)\right) = O\left( e^{-CN^{2/7}}\right).$$
    This completes the proof of Lemma \ref{crit_prob_bd}.
\end{proof}

\begin{thebibliography}{99}
    \bibitem{bhowal_mukh}
    S. Bhowal and S. Mukherjee, {\it Limit Theorems and Phase Transitions in the Tensor Curie-Weiss Potts Model}, {\tt arXiv:2307.01052}, 2023.

    \bibitem{cellular}
    S. E. M. Boas, Y. Jiang, R.M.H. Merks, S. A. Prokopiou and E.G. Rens, {\it Cellular Potts Model: Applications to Vasculogenesis and Angiogenesis}, {  Probabilistic Cellular Automata}, 27, 279--310, 2018.

    \bibitem{bornholdt}
    S. Bornholdt, {\it A $q$-spin Potts model of markets: Gain-loss asymmetry in stock indices as an emergent phenomenon}, {\tt arXiv:2112.06290}, 2021.

    \bibitem{socialsci}
    C. Bosconti, A. Corallo, L. Fortunato, A. A. Gentile, A. Massafra and P. Pell\`e, {\it Reconstruction of a Real World Social Network using the Potts Model and Loopy Belief Propagation}, { Frontiers in Psychology}, Vol. 6, 2015.

    \bibitem{moderate_deviation_Can}
    V. H. Can and V.H. Pham, \textit{A Cramér type moderate deviation theorem for the critical Curie-Weiss model}, Electronic Communications in Probability, Vol. 22, 1--12,  2017.

    \bibitem{impro}
    G. Celeux, F. Forbes and N. Peyrard, {\it EM-based image segmentation using Potts models
            with external field}, {Research Report RR-4456, INRIA}, inria-00072132, 2002.

    \bibitem{chenshao}
    H.Y. Louis Chen, X. Fang and Q. Shao, {\it From Stein identities to moderate deviations}, {The Annals of Probability}, Vol. 41, No. 1, 262--293, 2013.

    \bibitem{contucci}
    P. Contucci, E. Mingione and G. Osabutey, \textit{Limit theorems for the cubic mean-field Ising model},  arXiv:2303.14578, 2023.

    \bibitem{be_eich}
    S. Dommers and P. Eichelsbacher, \textit{Berry-Esseen bounds in the inhomogeneous Curie-Weiss
        model with external field}, Stochastic Processes and their Applications, Vol. 130 (2), 605--629, 2020.

    \bibitem{eichcubic}
    P. Eichelsbacher, \textit{Stein's method and a cubic mean-field model},  arXiv:2404.07587, 2024.

    \bibitem{eichelsbacher2010stein}
    P. Eichelsbacher and M. Löwe, {\it Stein's method for dependent random variables occurring in statistical mechanics}, Electronic Journal of Probability, Vol. 15, 962--988, 2010.

    \bibitem{eichelsbacher2015rates}
    P. Eichelsbacher and B. Martschink, {\it On rates of convergence in the Curie-Weiss-Potts model with an external field}, Annales de l'IHP Probabilités et statistiques, Vol. 51, No. 1, pp. 252--282, 2015.

    \bibitem{isingorig}
    E. Ising, {\it Beitrag zur theorie des ferromagnetismus}, {Zeitschrift f\"{u}r Physik}, 31:253--258, 1925.

    \bibitem{impro2}
    A. L .M. Levada, N. D. A. Mascarenhas and A. Tann\'us, {\it Pseudo-likelihood equations for Potts model on higher-order
            neighborhood systems: A quantitative approach for
            parameter estimation in image analysis}, { Brazilian Journal of Probability and Statistics}, Vol. 23, No. 2, 120--140, 2009.

    \bibitem{bmart}
    B. Martschink, \textit{Bounds on convergence for the empirical vector of the Curie-Weiss Potts model with a
        non-zero external field vector},  Statistics \& Probability Letters, Vol. 88, 118--126, 2014.

    \bibitem{gene77}
    E.V. Moltchanova,  J. Pitk\"aniemi and L. Haapala, {\it Potts model for haplotype associations}, {BMC Genet 6 (Suppl 1)}, S64, 2005.

    \bibitem{smfl}
    S. Mukherjee, J. Son and B. Bhattacharya, {\it Fluctuations of the Magnetization in the p-Spin Curie-Weiss Model}, {Communications in Mathematical Physics}, Vol. 387, Issue 2, 681--728, 2021.

    \bibitem{smmpl}
    S. Mukherjee, J. Son and B. Bhattacharya, {\it Estimation in Tensor Ising Models}, {Information and Inference: A Journal of the IMA}, Vol. 11, Issue 4, 1457--1500, 2022.

    \bibitem{smstein}
    S. Mukherjee, T. Liu and B. Bhattacharya, \textit{Moderate Deviation and Berry-Esseen Bounds in the p-Spin Curie-Weiss Model},  arXiv:2403.14122, 2024.

    \bibitem{Godwin}
    G. Osabutey, \textit{Phase properties of the mean-field Ising model with three-spin interaction}, Journal of Mathematical Physics, Vol. 65, 2024.

    \bibitem{pottsorig}
    R.B. Potts, {\it Some generalized order-disorder transformations}, {Mathematical Proceedings of the Cambridge Philosophical Society}, 48 (1): 106--109, 1952.

    \bibitem{reinert2009multivariate}
    G. Reinert and A. Röllin, {\it Multivariate normal approximation with Stein's method of exchangeable pairs under a general linearity condition}, The Annals of Probability, Vol. 37, No. 6, 2150--2173, 2009.

    \bibitem{stein2004use}
    C. Stein, P. Diaconis, S. Holmes and G. Reinert, {\it Use of exchangeable pairs in the analysis of simulations}, Lecture Notes-Monograph Series, 1--26, 2004.

    \bibitem{fina66}
    T. Takaishi, {\it Simulations of Financial Markets in a Potts-like Model}, {International Journal of Modern Physics C}, 16 (8), 2005.

    \bibitem{fywu}
    F. Y. Wu, {\it The Potts Model}, {Rev. Modern Phys. }, 54 (1), 235-268, 1982.

    \bibitem{spatstat}
    M. Zukovic, {\it Simulations of Environmental Spatial Data Using Ising and Potts Models}, {Conference: SigmaPhi}, Kolympari, Greece, 2008.

\end{thebibliography}
\end{document}